\documentclass[leqno,11pt]{article}

\newif\ifpaper
\newif\ifarxiv

\papertrue
\arxivtrue

\usepackage[utf8]{inputenc}
\usepackage[T1]{fontenc}
\usepackage{microtype}

\ifpaper
  \usepackage[
    letterpaper
  ]{geometry}
\else
  \usepackage[
    a5paper,
    margin=1.5em,
    includefoot
  ]{geometry}
\fi

\usepackage{amsmath}
\usepackage{amsthm}
\usepackage{amssymb}

\usepackage[english]{babel}
\usepackage{csquotes}

\ifarxiv
  \usepackage{ae,aecompl}
  \usepackage{times}
\else
  \usepackage{newtxtext,newtxmath}
\fi

\ifarxiv
  \usepackage[colorlinks,pagebackref,hypertexnames=false]{hyperref}
  \usepackage[numbers]{natbib}
  \newcommand{\cites}[1]{\cite{#1}}
\else
  \usepackage[
    backend=biber,
    hyperref=true,
    backref=true,
    isbn=false,
    doi=false,
    natbib=true,
    eprint=false,
    useprefix=true,
    maxcitenames=99, 
    maxbibnames=99,
    safeinputenc
  ]{biblatex}
  \usepackage[colorlinks,hypertexnames=false]{hyperref}
\fi

\usepackage{esint} 

\usepackage[yyyymmdd]{datetime}

\newcommand{\printreferences}{
  \ifarxiv        
    \bibliographystyle{plainnat}
    \bibliography{aux/refs}
   \else
    \printbibliography[heading=bibintoc]
  \fi
}

\newcommand{\defined}[2][\key]{%
  \def\key{#2}%
  \textbf{#2}%
  \index{#1}%
}

\newcommand{\N}{\bN}
\newcommand{\Z}{\bZ}

\newcommand{\R}{\bR}
\newcommand{\C}{\bC}

\newcommand{\su}{\mathfrak{su}}

\newcommand{\gl}{\mathfrak{gl}}

\newcommand{\U}{{\rm U}}
\newcommand{\PU}{{\bP\U}}

\DeclareMathOperator{\Ad}{Ad}

\DeclareMathOperator{\Gr}{Gr}

\DeclareMathOperator{\Hom}{Hom}

\DeclareMathOperator{\ad}{ad}
\DeclareMathOperator{\codim}{codim}
\DeclareMathOperator{\rk}{rk}

\DeclareMathOperator{\tr}{tr}

\newcommand{\End}{\mathop{\rm End}\nolimits}

\newcommand{\sEnd}{\mathop{\sE nd}\nolimits}

\renewcommand{\det}{\operatorname{det}}

\newcommand{\Bl}{\mathrm{Bl}}

\newcommand{\delbar}{\bar{\del}}
\newcommand{\del}{\partial}

\newcommand{\id}{{\rm id}}
\newcommand{\loc}{{\rm loc}}

\newcommand{\sing}{{\rm sing}}
\newcommand{\vol}{{\rm vol}}

\renewcommand{\P}{\bP}
\renewcommand{\epsilon}{\varepsilon}

\def\({\left(}
\def\){\right)}
\def\<{\left\langle}
\def\>{\right\rangle}

\newcommand{\co}{\mskip0.5mu\colon\thinspace}

\newcommand{\iso}{\cong}

\newcommand{\qand}{\quad\text{and}}

\newcommand{\qandq}{\quad\text{and}\quad}

\usepackage{mathtools}
\DeclarePairedDelimiter{\Abs}{\|}{\|}
\DeclarePairedDelimiter{\abs}{\lvert}{\rvert}
\DeclarePairedDelimiter\ceil{\lceil}{\rceil}
\DeclarePairedDelimiter{\set}{\lbrace}{\rbrace}
\DeclarePairedDelimiter{\braket}{\langle}{\rangle}

\newcommand{\inner}[2]{\braket{#1, #2}}
\newcommand{\winner}[2]{\braket{#1 \wedge #2}}

\def\MSC#1{\href{http://www.ams.org/msc/msc2010.html?&s=#1}{#1}}

\makeatletter
\renewcommand\xleftrightarrow[2][]{%
  \ext@arrow 9999{\longleftrightarrowfill@}{#1}{#2}}
\newcommand\longleftrightarrowfill@{%
  \arrowfill@\leftarrow\relbar\rightarrow}
\makeatother

\newcommand{\Gtwo}{G_2}

\usepackage{mathrsfs}

\newcommand{\rd}{{\rm d}}

\newcommand{\bC}{\mathbf{C}}

\newcommand{\bN}{\mathbf{N}}

\newcommand{\bP}{\mathbf{P}}

\newcommand{\bR}{\mathbf{R}}

\newcommand{\bZ}{\mathbf{Z}}

\newcommand{\cL}{\mathcal{L}}

\newcommand{\sA}{\mathscr{A}}

\newcommand{\sE}{\mathscr{E}}
\newcommand{\sF}{\mathscr{F}}

\newcommand{\sL}{\mathscr{L}}
\newcommand{\sM}{\mathscr{M}}

\newcommand{\fe}{{\mathfrak e}}

\newcommand{\fu}{{\mathfrak u}}

\newcommand{\fK}{{\mathfrak K}}
\newcommand{\fL}{{\mathfrak L}}

\usepackage{slashed}

\numberwithin{equation}{section}

\renewcommand{\eqref}[1]{\hyperref[#1]{\rm(\ref*{#1})}}

\usepackage{aliascnt}

\def\makeautorefname#1#2{\AtBeginDocument{\expandafter\def\csname#1autorefname\endcsname{#2}}}

\newcommand{\mynewtheorem}[2]{
  \newaliascnt{#1}{equation}          
  \newtheorem{#1}[#1]{#2}
  \aliascntresetthe{#1}
  \makeautorefname{#1}{#2}
}

\newcommand{\mynewproblem}[2]{
  \newaliascnt{#1}{myProblem}          
  \newtheorem{#1}[#1]{#2}
  \aliascntresetthe{#1}
  \makeautorefname{#1}{#2}
}

\mynewtheorem{theorem}{Theorem}
\newtheorem*{theorem*}{Theorem}
\mynewtheorem{prop}{Proposition}
\newtheorem*{prop*}{Proposition}

\mynewtheorem{cor}{Corollary}
\mynewtheorem{construction}{Construction}
\mynewtheorem{lemma}{Lemma}
\mynewtheorem{conjecture}{Conjecture}
\newtheorem*{conjecture*}{Conjecture}
\mynewtheorem{hyp}{Hypothesis}
\newtheorem{step}{Step}
\newtheorem{substep}{Step}
\numberwithin{substep}{step}
\makeautorefname{step}{Step}
\makeautorefname{substep}{Step}

\numberwithin{subcase}{case}
\makeautorefname{case}{Case}
\makeautorefname{subcase}{case}

\theoremstyle{remark}
\mynewtheorem{remark}{Remark}
\newtheorem*{remark*}{Remark}

\theoremstyle{definition}
\mynewtheorem{definition}{Definition}
\newtheorem*{definition*}{Definition}
\mynewtheorem{example}{Example}
\mynewtheorem{exercise}{Exercise}
\mynewproblem{problem}{Problem}
\mynewtheorem{solution}{Solution}
\mynewtheorem{convention}{Convention}
\newtheorem*{convention*}{Convention}
\newtheorem*{conventions*}{Conventions}
\mynewtheorem{question}{Question}
\newtheorem*{question*}{Question}

\makeautorefname{table}{Table}        
\makeautorefname{chapter}{Chapter}
\makeautorefname{section}{Section}
\makeautorefname{subsection}{Section}
\makeautorefname{subsubsection}{Section}
\makeautorefname{footnote}{Footnote}

\author{
  Adam Jacob \\
  UC Davis
  \and
  Thomas Walpuski \\
  Massachusetts Institute of Technology
}
\title{
  Hermitian Yang--Mills metrics on reflexive sheaves over asymptotically cylindrical Kähler manifolds}
\date{2017-01-27}

\begin{document}

\maketitle

\begin{abstract}
  We prove an analogue of the Donaldson--Uhlenbeck--Yau theorem for asymptotically cylindrical Kähler manifolds:
  If $\sE$ is a reflexive sheaf over an ACyl Kähler manifold, which is asymptotic to a $\mu$--stable holomorphic vector bundle,
  then it admits an asymptotically translation-invariant projectively Hermitian Yang--Mills metrics (with curvature in $L^2_\loc$ across the singular set).
  Our proof combines the analytic continuity method of \citet{Uhlenbeck1986} with the geometric regularization scheme introduced by \citet{Bando1994}.
\end{abstract}

{\small
  \noindent
  \textbf{Keywords:}
    \textit{holomorphic vector bundles, reflexive sheaves, Hermitian Yang--Mills metrics, asymptotically cylindrical Kähler manifolds} \\
  \medskip
  \noindent
  \textbf{MSC2010:}
    \MSC{53C07}; \MSC{14J60}, \MSC{32Q15}
}


\section{Introduction}
\label{Sec_Introduction}

In this paper we construct (singular) projectively Hermitian Yang--Mills ($\P$HYM) metrics over a certain class of complete non-compact Kähler manifolds.

In the compact case this problem has been extensively studied.
Its solution provides a particularly beautiful example of the relation between canonical metrics and algebro-geometric notions of stability:
a holomorphic vector bundle over a compact  Kähler admits a $\P$HYM metric if and only if is $\mu$--polystable.
This was first proved for curves by \citet{Narasimhan1965},
for algebraic surfaces by \citet{Donaldson1985},
and for arbitrary compact Kähler manifolds by \citet{Uhlenbeck1986}.

It is an interesting and important question to ask:
under which hypothesis does a holomorphic vector bundle over a complete non-compact Kähler manifolds admit a $\P$HYM metric?\footnote{%
  This question was also raised in Yau's 2015 Shanks Lecture \cite[p.~66]{Yau2015}.}
The answer to this question is not completely understood, but a number of partial results have been obtained.
For asymptotically conical Kähler manifolds, Bando proved the existence of $\P$HYM metrics on holomorphic vector bundles which are flat at infinity \cite{Bando1993}.
\citet{Ni2001} proved that a holomorphic vector bundle over a complete non-compact Kähler manifold with a spectral gap admits a $\P$HYM metric if and only if it admits a metric whose failure to be $\P$HYM is in $L^p$ for $p > 1$  (using an argument similar to Donaldson's solution of the Dirichlet problem for the $\P$HYM equation \cite{Donaldson1992}).
\citet{Ni2002} showed that the same conclusion holds, for example, if the Kähler manifold satisfies a $L^2$ Sobolev inequality and $p \in [1,n/2)$, or if it is non-parabolic (i.e., admits a positive Green's function) and $p=1$.

\paragraph{Main result}

In this article we concentrate on the asymptotically cylindrical case, and in view of the applications we have in mind we work with reflexive sheaves (not just holomorphic vector bundles).

\begin{theorem}
  \label{Thm_ACylDUY}
  Let $V$ be an asymptotically cylindrical (ACyl) Kähler manifold with asymptotic cross-section $D$.
  Let $\sE_D$ be a $\mu$--stable vector bundle over $D$, and
  $\sE$ a reflexive sheaf asymptotic to $\sE_D$.

  In this situation there exists an asymptotically translation-invariant Hermitian metric $H$ on $\sE$ which satisfies the projective Hermitian Yang--Mills ($\P$HYM) equation
  \begin{equation}
    \label{Eq_PHYM}
    K_{H} 
    :=
      i\Lambda F_{H} - \frac{\tr(i\Lambda F_{H})}{\rk \sE}\cdot \id_\sE
    =
      0,
  \end{equation}
  and $|F_H| \in L^2_\loc(V)$.
\end{theorem}

\begin{remark}
  \label{Rmk_HYMvPHYM}
  A $\P$HYM metric $H$ on $\sE$ is Hermitian Yang--Mills (HYM) if and only if the induced metric $h$ on $\det \sE$ is HYM, that is, $i\Lambda F_h = \frac{\tr(i\Lambda F_{H})}{\rk \sE}$ is constant.
  Every asymptotically translation-invariant line bundle over an ACyl Kähler manifold has a HYM metric;
  however, this metric will typically not be asymptotically translation invariant.
  See \autoref{Sec_HYMLineBundles} for a detailed discussion.
\end{remark}

\begin{remark}
  The definition of asymptotically cylindrical Kähler manifolds we work with is given in \autoref{Def_ACylKahler};
  it includes being asymptotically fibred.
\end{remark}

\begin{remark}
  \label{Rmk_SaEarp}
  The question of the existence of HYM metrics on holomorphic bundles (with trivial determinant) over ACyl Calabi--Yau $3$--folds was studied earlier by Sá Earp \cite{SaEarp2011} (using the Yang--Mills heat flow).
  Our result improves on his in that we consider general ACyl Kähler manifolds and handle reflexive sheaves;
  moreover, we give a complete proof of the exponential decay to a $\P$HYM metric over $D$ (which is crucial for applications).
\end{remark}

\begin{remark}
  In dimension four, there is prior work on the relation between ASD instantons and holomorphic vector bundles over cylindrical manifolds by Guo \cite{Guo1996} and Owens \cite{Owens2001}.
\end{remark}

\paragraph{Examples and applications}

There are plenty of examples of ACyl Kähler manifolds and reflexive sheaves on them.
Given any smooth projective variety $Z$ containing a smooth divisor $D$ and fibred by $\abs{D}$, $V := Z \setminus D$ can be given the structure of an ACyl Kähler manifold \cite[Section 4.2, Part 1]{Haskins2012}.
\autoref{Thm_ACylDUY} can be applied to any holomorphic vector bundle $\sE$ on $Z$ such that $\sE|_D$ is $\mu$--stable.
One often wants to construct $\sE$ by extending a holomorphic vector bundle $\sE_D$ on $D$ to all of $Z$.
This can always be achieved with $\sE$ being a reflexive sheaf---by first extending $\sE_D$ as a torsion-free sheaf and then taking the reflexive hull.
Whether or not this extension can be arranged to be a holomorphic vector bundle is a subtle question.
This is one of the reasons why it is desirable to allow reflexive sheaves.

ACyl Calabi--Yau $3$--folds are an important ingredient in the construction of twisted connected sum $\Gtwo$--manifolds \cite{Kovalev2003,Kovalev2011,Corti2012a}.
Building on \cite{SaEarp2011}, Sá Earp and the second named author gave a construction of a class of Yang--Mills connections, called $\Gtwo$--instantons, over such twisted connected sums \cite{SaEarp2013};
see \cite{Walpuski2015} for a concrete example.
We hope that the current work will be a first step towards the construction of singular $\Gtwo$--instantons on twisted connected sums.
$\Gtwo$--instantons play a central role in Donaldson and Thomas' vision of gauge theory in higher dimensions \cite{Donaldson1998},
and understanding singularities and their formation is an important part of making their ideas rigorous;
see, e.g., \cites{Walpuski2013,Walpuski2013a,Haydys2014}.

\paragraph{Proof idea}

We first prove \autoref{Thm_ACylDUY} for holomorphic vector bundles. 
After a suitiable choice of an initial Hermitian metric $H_0$ on $\sE$, we construct a $\P$HYM metric using the Uhlenbeck--Yau continuity method.
The crucial point is the a priori $C^0$ estimate on the endomorphism $s$ relating $H_0$ and the Hermitian metric $H_t=H_0e^{s}$ along the continuity path.
Unlike in \cites{Bando1993,Ni2002}, a solution to the Poisson equation $\Delta f=\abs{K_{H_0}}$ can not act as a barrier, since on $V$ such a solution does not have exponential decay---in fact, it decreases linearly along cylindrical end. 
Instead, we use an adaptation to our setup of Sá Earp's argument in \cite{SaEarp2011}:
his proof first exploits the barrier to obtain a bound of the form $\Abs{s}_{L^\infty}^3 \lesssim \Abs{s}_{L^2}^2$, and then uses the Donaldson functional on transverse slices along the cylindrical end to show that $\Abs{s}_{L^2} \lesssim \Abs{s}_{L^\infty}$.
Besides the construction of the initial Hermitian metric $H_0$, this is the crucial point at which $\mu$--stability enters into the proof.
To prove a priori exponential decay bounds we use ideas of Haskins, Hein and Nordström \cite{Haskins2012}.

Once \autoref{Thm_ACylDUY} is established for holomorphic vector bundles, we prove the general case for a reflexive sheaf $\sE$ following a geometric regularization scheme, introduced by \citet{Bando1994}, based on approximating $\sE$ and $V$ by a holomorphic vector bundle and a family of ACyl Kähler metrics on a blow-up of $V$.
The main difficulty is controlling the barrier $f$ as the metrics degenerate.
Once $f$ is controlled, the $C^0$ bound on compact subsets away from the singular set of $\sE$ follows, and the arguments from the holomorphic vector bundle case can be applied directly.

\paragraph{Conventions}

We denote by $c > 0$ a generic constant, which depends only on $V$, $\sE$, and the reference metric $H_0$ constructed in \autoref{Sec_UYContinuityMethod}.
Its value might change from one occurrence to the next.
Should $c$ depend on further data we indicate this by a subscript.
We write $x \lesssim y$ for $x \leq c y$ and $x \asymp y$ for $c^{-1}y \leq x \leq cy$.
$O(x)$ denotes a quantity $y$ with $\abs{y} \lesssim x$.


\section{ACyl Kähler manifolds}
\label{Sec_ACylKahler}

In this section we briefly introduce some notation,
recall the necessary linear analysis,
and provide the detail promised in \autoref{Rmk_HYMvPHYM}.

\begin{definition}
  \label{Def_ACylKahler}
  Let $(D,g_D,I_D)$ be a compact Kähler manifold.
  A Kähler manifold $(V,g,I)$ is called \defined{asymptotically cylindrical (ACyl)} with asymptotic cross-section $(D,g_D,I_D)$ if there exists a constant $\delta_V > 0$, a compact subset $K \subset V$ and a diffeomorphism $\pi \co V\setminus K \to (1,\infty) \times S^1 \times D$ such that
  \begin{equation*}
    \abs{\nabla^k (\pi_*g - g_\infty)}
    +
    \abs{\nabla^k (\pi_*I - I_\infty)}
    =
    O(e^{-\delta_V \ell}),
  \end{equation*}
  for all $k \in \N_0$, with
  \begin{equation*}
    g_\infty
    :=
    \rd \ell^2\oplus \rd\theta^2 \oplus g_D
    \qandq
    I_\infty
    =
    \begin{pmatrix}
      0 & -1 \\
      1 & 0
    \end{pmatrix}
    \oplus
    I_D.
  \end{equation*}
  Here $(\ell,\theta)$ are the canonical coordinates on $(0,\infty)\times S^1$.
  Moreover, we assume that the map $V\setminus K \to (0,\infty) \times S^1$ is holomorphic.
\end{definition}

In what follows, we suppose an ACyl Kähler manifold $V$ with asymptotic cross-section $D$ has been fixed.
By slight abuse of notation we denote by $\ell\co V \to [0,\infty)$ a smooth extension of $\ell\circ \pi\co V\setminus K \to (1,\infty)$ such that $\ell \leq 1$ on $K$.
Given $L > 1$, we define the truncated manifold
\begin{equation*}
  V_L := \ell^{-1}([0,L]).
\end{equation*}
Given $z = (L,\theta) \in (1,\infty)\times S^1$, we set
\begin{equation}
  \label{Eq_Dz}
  D_z := \pi^{-1}(\set{(L,\theta)}\times D).
\end{equation}

\subsection{Reflexive sheaves and Hermitian metrics}
\label{Sec_ReflexiveHermitian}
 
\begin{definition}
  \label{Def_ATIReflexiveSheaf}
  Let $\sE_D = (E_D,\delbar_D)$ be a holomorphic vector bundle over $D$.
  Let $\sE$ be a reflexive sheaf over $V$ with singular set $S := \sing(\sE)$ and underlying smooth vector bundle $E \to V\setminus S$.
  We say that $\sE$ is \defined{asymptotic} to $\sE_D$ if the following hold:
  \begin{itemize}
  \item 
    There exists a constant $L_0 \geq 2$ such that $S \subset V_{L_0 - 1}$.
    In particular, $E|_{V\setminus V_{L_0}}$ has a $\delbar$--operator.
  \item
    There exists a bundle isomorphism $\bar\pi \co E|_{V\setminus V_{L_0}} \to E_\infty$ covering $\pi$ and a constant $\delta_\sE > 0$ such that
    \begin{equation*}
      \abs{\nabla^k (\bar\pi_*\delbar - \delbar_\infty)} = O(e^{-\delta_\sE\ell}),
    \end{equation*}
    for all $k \in \N_0$ and $\ell \geq L_0$. 
    Here $\sE_\infty = (E_\infty,\delbar_\infty)$ is the pullback of $\sE_D = (E_D,\delbar_D)$ to $(L_0,\infty)\times S^1\times D$;
    moreover, we have chosen an auxiliary Hermitian metric on $E_D$ and pulled it back to $E_\infty$.\footnote{%
      The definition is insensitive to the precise choice, since $D$ is compact.}
  \end{itemize}
  We say that $(\sE,\delbar)$ is \defined{asymptotically translation-invariant} if it is asymptotic to some holomorphic vector bundle over $D$.
\end{definition}

\begin{definition}
  Let $\sE$ be a reflexive sheaf over $V$ asymptotic to $\sE_D$.
  Let $H_D$ be a Hermitian metric on $E_D$.
  A \defined{Hermitian metric} on $\sE$ is a Hermitian metric $H$ on $\sE|_{V\setminus S}$.
  We say that it is \defined{asymptotic} to $H_D$ if there exist a constant $\delta_H > 0$ such that
  \begin{equation*}
    \abs{\nabla^k (\bar\pi_*H - H_\infty)} = O(e^{-\delta_H\ell})
  \end{equation*}
  for all $k \in \N_0$ and $\ell \geq L_0$.
  Here $H_\infty$ is the pullback of $H_D$ to $\sE_\infty$.
  (We take the background metric, used in the comparison, to be $H_\infty$.)
  We say that $H$ is \defined{asymptotically translation-invariant} if it is asymptotic to some Hermitian metric $H_D$.
\end{definition}

Given a Hermitian metric $H$ on a holomorphic vector bundle $(\sE,\delbar)$,
there exists a unique connection $A_H$,
called the \defined{Chern connection},
which preserves the Hermitian metric and satisfies $\nabla^{0,1}_{A_H} = \delbar$;
see, e.g., \cite[Theorem 3.18]{Ballmann2006}.
We denote the curvature of this connection by $F_H$.

\begin{definition}
  A Hermitian metric $H$ on a reflexive sheaf $\sE$ is called \defined{projectively Hermitian Yang--Mills ($\P$HYM)} if $K_H \in C^\infty(V\setminus S,i\su(E,H))$ defined by
  \begin{equation*}
    K_H := i\Lambda F_{H} - \frac{\tr(i\Lambda F_{H})}{\rk \sE} \cdot \id_\sE
  \end{equation*}
  vanishes.
\end{definition}

\subsection{Linear analysis}
\label{Sec_LinearAnalysis}

In the subsequent sections we need a few results about linear analysis on ACyl Kähler manifolds.
We will simply state the required results and sketch their proofs.
For a nice review of linear analysis on ACyl manifolds we refer the reader to \cite[Section 2.1]{Haskins2012};
see also Maz'ya and Plamenevski{\u\i} \cite{Mazya1978} and Lockhart and McOwen \cite{Lockhart1985}.

Fix a holomorphic vector bundle $\sE$ asymptotic to $\sE_D$ and a Hermitian metric $H$ asymptotic to $H_D$.

\begin{definition}
  For $k \in \N$, $\alpha \in (0,1)$ and $\delta \in \R$,
  define
  \begin{equation*}
    C^{k,\alpha}_\delta(V)
    :=
    \set*{
      f \in C^{k,\alpha}(V) : \Abs{f}_{C^{k,\alpha}_\delta} < \infty
    },
  \end{equation*}
  with
  \begin{equation*}
    \Abs{\cdot}_{C^{k,\alpha}_\delta} := \Abs{e^{\delta \ell}\cdot}_{C^{k,\alpha}},
  \end{equation*}
  and set
  \begin{equation*}
    C^\infty_\delta(V) := \bigcap_{k \in \N} C^{k,\alpha}_\delta(V).
  \end{equation*}
  Similarly, we define $C^{k,\alpha}_\delta(V,i\su(E,H))$ and $C^\infty_\delta(V,i\su(E,H))$.
\end{definition}

\begin{prop}
  \label{Prop_ScalarLaplacian}
  For $0 < \delta \ll_D 1$, the linear map $C^{k+2,\alpha}_\delta(V) \oplus \R \to C^{k,\alpha}_\delta(V)$ defined by
  \begin{equation*}
    (f,A) \mapsto \Delta f - A\Delta\ell
  \end{equation*}
  is an isomorphism.
\end{prop}

\begin{proof}
  This is \cite[Proposition 2.7]{Haskins2012} together with the observation that
  \begin{equation*}
    \int_V \Delta \ell = - \vol(S^1\times D).
    \qedhere
  \end{equation*}
\end{proof}

\begin{prop}
  \label{Prop_TraceFreeLaplacian}
  If $H_D$ is HYM, $\sE_D$ is simple and $\abs{\delta} \ll_{H_D} 1$, then the linear operator $\nabla_{H_0}^*\nabla_{H_0} \co C^{k+2,\alpha}_\delta(V,i\su(E,H)) \to C^{k,\alpha}_\delta(V,i\su(E,H))$ is Fredholm of index zero.
\end{prop}

\begin{proof}
  We use the theory explained in \cite[Section 2.1]{Haskins2012}.
  The linear operator $\nabla_{H_0}^*\nabla_{H_0}$ is asymptotic to the translation-invariant linear operator
  \begin{equation*}
    -\del_\ell^2 - \del_\theta^2 + \nabla_{H_D}^*\nabla_{H_D}
  \end{equation*}
  acting on sections of $i\su(E_\infty,H_\infty)$.
  Since $H_D$ is $\P$HYM,
  \begin{equation*}
    \frac12 \nabla_{H_D}^*\nabla_{H_D}
    = \del_{H_D}^*\del_{H_D}
    = \delbar_{\sE_D}^*\delbar_{\sE_D}.
  \end{equation*}
  The latter is invertible because $\sE_D$ is simple.
  Consequently, the spectrum of $-\del_\theta^2 + \nabla_{H_D}^*\nabla_{H_D}$ is contained in $[\lambda_D,\infty)$, for some $\lambda_D > 0$.
  This implies the Fredholm property for $\abs{\delta} < \sqrt{\lambda_D}$ by \cite[Proposition 2.4]{Haskins2012}.
  Since $\nabla_{H_D}^*\nabla_{H_D}$ is formally self-adjoint and $0$ is not a critical weight, the index is zero;
  cf.~\cite[Theorem 7.4]{Lockhart1985}.
\end{proof}

\subsection{Hermitian Yang--Mills metrics on line bundles}
\label{Sec_HYMLineBundles}

\begin{prop}
  \label{Prop_HYMLineBundles}
  Let $\sL$ be a line bundle asymptotic to $\sL_D$ and denote by $h_D$ a Hermitian metric on $\sL_D$ with
  \begin{equation*}
    i\Lambda F_{h_D}
    = \lambda
    := \frac{2\pi \cdot \deg(\sL_D)}{(n-2)! \cdot \vol(D)}.\footnote{Such
      a Hermitian metric exists and is unique up to multiplication by a positive constant.}
  \end{equation*}
  There exist a unique Hermitian metric $h_0$ asymptotic to $h_D$ and $A \in \R$ such that $h := h_0 e^{-A\ell}$ satisfies
  \begin{equation*}
    i\Lambda F_h = \lambda.
  \end{equation*}
\end{prop}

\begin{proof}
  Let $h_{-1}$ be any Hermitian metric asymptotic to $h_D$.
  We have
  \begin{equation*}
    \lambda - i\Lambda F_{h_{-1}} \in C^\infty_\delta(V).
  \end{equation*}
  By \autoref{Prop_ScalarLaplacian} there is a unique pair $f \in C^\infty_\delta(V)$ and $A \in \R$ such that
  \begin{equation*}
    \Delta(f-A\ell) = \lambda - i\Lambda F_{h_{-1}}.
  \end{equation*}
  The proposition follows with $h_0 := h_{-1}e^f$.
\end{proof}

The number $A(\sL)$ defined by \autoref{Prop_HYMLineBundles} is an invariant of the asymptotically translation-invariant line bundle $\sL$.
It can be computed as
\begin{equation*}
  A(\sL) := \frac{1}{\vol(S^1\times D)}\int_V \lambda - i\Lambda F_h
\end{equation*}
with $h$ denoting \emph{any} Hermitian metric asymptotic to some $h_D$ as in \autoref{Prop_HYMLineBundles}.
It is closely related to the first Chern class:
if $\sL_1$ and $\sL_2$ are both asymptotic to $\sL_D$, then $c_1(\sL_1) - c_1(\sL_2) \in H^2_c(V)$ and
\begin{equation*}
  A(\sL_1) - A(\sL_2) =
    \frac{
      2\pi \cdot \inner{\(c_1(\sL_1)-c_1(\sL_2)\)\cup[\omega]^{n-1}}{[V]}
    }{
      (n-1)!\cdot\vol(S^1\times D)
    }.
\end{equation*}

It follows from the above that $\sE$ as in \autoref{Thm_ACylDUY} admits an asymptotically translation invariant HYM metric if and only if $A(\det \sE) = 0$.


\section{The Uhlenbeck--Yau continuity method}
\label{Sec_UYContinuityMethod}

In this section we begin the proof of \autoref{Thm_ACylDUY} in the  case when $\sE$ is a holomorphic vector bundle.
We use the continuity method introduced by \citet{Uhlenbeck1986}; see also Lübke and Teleman's beautiful books \cites{Luebke1995,Luebke2006}.

We fix some
\begin{equation*}
  0 < \delta <  \min \set*{ \delta_V, \delta_\sE, \sqrt{\lambda_D} }
\end{equation*}
and will shortly construct a background Hermitian metric $H_0$ on $\sE$ which is asymptotically translation-invariant and satisfies
\begin{equation}
  \label{Eq_KH0Decay}
  K_{H_0} \in C^\infty_\delta(V,i\su(E,H_0)).
\end{equation}
Given such an $H_0$, we define a map 
\begin{equation*}
  \fL \co
    C^\infty_\delta(V,i\su(E,H_0)) \times [0,1]
    \to
    C^\infty_\delta(V,i\su(E,H_0))
\end{equation*}
by
\begin{equation*}
  \fL({s},t) :=
    \Ad(e^{{s}/2})K_{H_0e^{s}}
    +
    t\cdot {s}.
\end{equation*}
Set
\begin{equation*}
  I :=
    \set*{
      t \in [0,1] :
      \fL({s},t) = 0
      \text{ for some }
      {s} \in C^\infty_\delta(V,i\su(E,H_0))
    }.
\end{equation*}
We will show that $1 \in I$, $I \cap (0,1]$ is open and $I$ is closed; hence, $I = [0,1]$.
Since $\fL(s,0)=0$ precisely means that $H=H_0e^s$ satisfies \eqref{Eq_PHYM}, this will prove \autoref{Thm_ACylDUY} when $\sE$ is a holomorphic vector bundle.

\begin{prop}
  \label{Prop_LuebkeTelemanTrick}
  There exists an asymptotically translation-invariant Hermitian metric $H_0$ on $\sE$ satisfying \eqref{Eq_KH0Decay}, and there exists an $s \in C^\infty_\delta(V,i\su(E,H_0))$ such that $\fL({s},1) = 0$.
\end{prop}

\begin{proof}
  We use a trick discovered by \citet[Lemma 3.2.1]{Luebke1995}.
  By the Donaldson--Uhlenbeck--Yau theorem \cites{Donaldson1985,Uhlenbeck1986,Donaldson1986} there exists a $\P$HYM metric $H_D$ on $\sE_D$.
  One can easily construct a Hermitian metric $H_{-1}$ asymptotic to $H_D$ (at rate $\delta_{H_{-1}} = \delta$) which satisfies
  \begin{equation*}
    \kappa := K_{H_{-1}} \in C^\infty_\delta(V,i\su(E,H_{-1})).
  \end{equation*}

  The Hermitian metric
  \begin{equation*}
    H_0 := H_{-1}e^{\kappa}
  \end{equation*}
  is asymptotic to $H_D$ (at rate $\delta_{H_0} = \delta$);
  moreover, we have \eqref{Eq_KH0Decay},
  and $\kappa \in C^\infty_\delta(V,i\su(E,H_0))$ satisfies
  \begin{equation*}
    \fL(-\kappa,1)
    = \Ad(e^{-\kappa/2})(K_{H_{-1}}) - \kappa
    = 0. 
    \qedhere
  \end{equation*}
\end{proof}


\section{Linearising \texorpdfstring{$\fL = 0$}{L = 0}}
\label{Sec_Linearisation}

Having just proved that $1 \in I$, the next step is to show that $I \cap (0,1]$ is open.
This will be established in this section by linearising the equation $\fL = 0$.

Since
\begin{equation*}
  \fL(s,t)
  = \Ad(e^{s/2})\(K_{H_0} + i\Lambda\delbar(e^{-s}\del_{H_0}e^s)\) + t\cdot s,
\end{equation*}
it extends to a smooth map
\begin{equation*}
  \fL\co C^{2,\alpha}_\delta(V,i\su(E,H_0)) \times [0,1] \to C^{0,\alpha}_\delta(V,i\su(E,H_0)).
\end{equation*}
The fact that $I \cap (0,1]$ is open is an immediate consequence of the following two propositions and the implicit function theorem for Banach spaces;
see, e.g., \cite[Theorem A.3.3]{McDuff2012}.

\begin{prop}
  \label{Prop_Invertibility}
  If $(s,t) \in C^{2,\alpha}_\delta(V,i\su(E,H_0)) \times (0,1]$ is a solution of $\fL(s,t) = 0$,
  then the linearisation
  \begin{equation*}
    L_{s,t} := \frac{\rd \fL}{\rd s}(s,t) \co C^{2,\alpha}_\delta(V,i\su(E,H_0)) \to C^{0,\alpha}_\delta(V,i\su(E,H_0))
  \end{equation*}
  is invertible.
\end{prop}

\begin{prop}
  \label{Prop_Regularity}
  If $(s,t) \in C^{2,\alpha}_\delta(V,i\su(E,H_0)) \times [0,1]$ is a solution of $\fL(s,t) = 0$,
  then $s \in C^\infty_\delta(V,i\su(E,H_0))$.
\end{prop}

The proofs of both of these results are essentially identical to those of the analogous results in the compact setting; see \cite[Lemma 4.6 and Lemma 4.8]{Luebke2006}.
The proofs make use of the explicit formulae for $\Ad(e^{s/2})K_{H_0e^s}$ and its derivative in the direction of $s$.
The derivation of these, while rather straight-forward, is somewhat tedious and therefore relegated to \autoref{Sec_UsefulFormulae}.

\begin{proof}[Proof of \autoref{Prop_Regularity}]
  By \autoref{Prop_KH0ExpS}, the equation $\fL(s,t) = 0$ is equivalent to
  \begin{equation*}
    \(\frac12\nabla_{H_0}^*\nabla_{H_0}+t\) s
    + B(\nabla_{H_0} s\otimes \nabla_{H_0} s) 
    = C(K_{H_0}).
  \end{equation*}
  where $B$ and $C$ are linear with coefficients depending on $s$, but not on its derivatives.
  The result now follows from a standard elliptic bootstrapping procedure.
\end{proof}

\begin{proof}[Proof of \autoref{Prop_Invertibility}]
  By \autoref{Prop_DerivativeKH0ExpS}, the linear operator $L_{s,t}$ is given by
  \begin{equation*}
    L_{s,t}\hat s
    = \frac12\nabla_{\tilde A_s}^*\nabla_{\tilde A_s} \Ad(e^{s/2})\Upsilon(-s)\hat s + t \hat s
  \end{equation*}
  with $\Upsilon$ as defined in \eqref{Eq_Upsilon}.
  The linear operator $L_{s,t}$ can be joined to $\frac12\nabla_{H_0}^*\nabla_{H_0} + t$ by a path of bounded linear operators which are asymptotic to $\frac12$--times $-\del_\ell^2 - \del_\theta^2 + \nabla_{H_D}^*\nabla_{H_D} + 2t$.
  The argument in the proof of \autoref{Prop_TraceFreeLaplacian} shows that this is a path of Fredholm operators.
  Therefore, the index of $L_{s,t}$ agrees with that of $\frac12\nabla_{H_0}^*\nabla_{H_0} + t$ and thus vanishes.
  By \autoref{Prop_UpsilonHatSS}, we have
  \begin{equation*}
    \int_V \inner{L_{s,t} \hat s}{\Ad(e^{s/2})\Upsilon(-s)\hat s}
    \geq t \int_V \abs{s}^2;
  \end{equation*}
  hence, $L_{s,t}$ has trivial kernel and thus is invertible.
\end{proof}


\section{A priori estimate}
\label{Sec_APrioriEstimate}

Given the following a priori estimate, it is an immediate consequence of Arzelà--Ascoli that $I$ is closed.

\begin{prop}
  \label{Prop_APrioriEstimate}
  If $(s,t) \in C^\infty_\delta(V,i\su(E,H_0)) \times [0,1]$ satisfies $\fL(s,t) = 0$,
  then
  \begin{equation*}
    \Abs{s}_{C^{k,\alpha}_\delta} \leq c_{k,\alpha}.
  \end{equation*}
\end{prop}

The proof of this proposition, to which this section is devoted, has two steps:
First we prove that $\Abs{s}_{C^0}$ is bounded by a constant depending only on $H_0$ using ideas from \cite{SaEarp2011}.
This implies that $\Abs{s}_{C^k}$ is bounded by a constant depending only on $k$ and $H_0$ by an argument of Bando and Siu \cite[Proposition 1]{Bando1994}.
(For the reader's convenience we give a detailed proof of this in \autoref{Sec_BSInteriorEstimate}.)
The second step is a decay estimate which is similar to \cite[Steps 3 and 4 in the proof of Theorem 4.1]{Haskins2012}.

\subsection{A priori \texorpdfstring{$C^k$}{Ck} estimate}

\begin{prop}
  \label{Prop_APrioriCkEstimate}
  If $(s,t) \in C^\infty_\delta(V,i\su(E,H_0)) \times [0,1]$ satisfies $\fL(s,t) = 0$,
  then
  \begin{equation*}
    \Abs{s}_{C^k} \leq c_k.
  \end{equation*}
\end{prop}

\begin{proof}
  By \autoref{Thm_BSInteriorEstimate} it suffices to prove the proposition for $k = 0$.
  Fix $L_0 \gg 1$ and set
  \begin{equation*}
    N := \Abs{s}_{L^\infty(V)} \qandq
    M := \Abs{s}_{L^\infty(V\setminus V_{L_0})}.
  \end{equation*}
  
  \begin{step}
    We have
    \begin{equation*}
      N - M \lesssim L_0+1.
    \end{equation*}
  \end{step}

  We can assume that $\abs{s}$ achieves its maximum at $x_0 \in V_{L_0}$.
  From \autoref{Prop_LaplacianIdentity} and $\fL(s,t) = 0$ it follows that 
  \begin{equation}
    \label{Eq_DeltaS2}
    \Delta \abs{s}^2 + 4t \abs{s}^2 \leq -4\inner{K_{H_0}}{s};
  \end{equation}
  hence,
  \begin{equation*}
    \Delta \abs{s}^2 \leq 4N \abs{K_{H_0}}.
  \end{equation*}
  Denote by $f \in C^{2,\alpha}_\delta(V)$ and $A > 0$ the unique solution to
  \begin{equation*}
    \Delta(f-A\ell) = 4\abs{K_{H_0}}.
  \end{equation*}
  Applying the maximum principle to $\abs{s}^2 - N(f - A\ell)$ on $V_{L_0}$ we have
  \begin{equation*}
    N^2 \leq M^2 + N(AL_0+ 2\Abs{f}_{L^\infty}).
  \end{equation*}
  This implies the assertion since $M \leq N$.

  \begin{step}
    \label{Step_SlicewiseKLowerBound}
    We have
    \begin{equation*}
      \sqrt{M}
      \lesssim
        \Abs{K_{H_0e^s|_{D_z}}}_{L^2(V\setminus V_{L_0})}.
    \end{equation*}
    Here $D_z$ is as in \eqref{Eq_Dz} for $z = (L,\theta) \in (L_0,\infty)\times S^1$.
  \end{step}

  \begin{substep}
    \label{Step_SlowDecay}
    If $x_0 \in \overline{V\setminus V_{L_0}}$ is such that
    \begin{equation*}
      \abs{s}(x_0) = M,
    \end{equation*}
    then for all $L \geq \ell(x_0)$ we have
    \begin{equation*}
      \Abs{s}_{L^\infty(\del V_L)} - \frac12 M
      \gtrsim
        \ell(x_0)-L.
    \end{equation*}
  \end{substep}

  By the maximum principle applied to $\abs{s}^2 - N(f - A\ell)$ on $V_L$ we have
  \begin{equation*}
    M^2 - Nf(x_0) + NA\ell(x_0)
    \leq
    \Abs{s}_{L^\infty(\del V_L)}^2
    + N\Abs{f}_{L^\infty(\del V_L)} + NAL.
  \end{equation*}
  The assertion follows since we can assume that $M \geq 8\Abs{f}_{L^\infty(V\setminus V_{L_0})}$ and $N \leq 2M$.

  \begin{substep}
    \label{Step_IntegralLowerBound}
    There are $L_0 \leq L_1 < L_2$ with $L_2-L_1 \asymp M$ such that
    \begin{equation*}
      M^{3/2} 
      \lesssim \Abs{s}_{L^2(V_{L_2} \setminus V_{L_1})}.
    \end{equation*}
  \end{substep}

  By \autoref{Step_SlowDecay} we have
  \begin{equation*}
    M \lesssim \Abs{s}_{L^\infty(\del V_{L})}
  \end{equation*}
  for $0 \leq L - \ell(x_0) \asymp M$;
  hence, using the mean value inequality \cite[Theorem 9.20]{Gilbarg2001} it follows that
  \begin{equation*}
    M^2 
    \lesssim \int_{V_{L+1}\setminus V_{L-1}} \abs{s}^2 + e^{-\delta L_0} M.
  \end{equation*}
  Since $L_0 \gg 1$, the second term on the right-hand side can be rearranged.
  Summing over $L-\ell(x_0) = 1, \ldots, k$ (with $k \asymp M$) yields the asserted inequality.

  \begin{substep}
    \label{Step_SlicewiseUpperBounds}
    We have
    \begin{equation*}
      \Abs{s}_{L^2(D_z)} - 1/2
      \lesssim
        M \Abs{K_{H_0e^s|_{D_z}}}_{L^2(D_z)}.
    \end{equation*}
  \end{substep}

  At this stage the $\mu$--stability of $\sE_D$ comes into play via the Donaldson functional $\sM$; see \autoref{Sec_DonaldsonFunctional}.
  Since $L_0 \gg 1$ and $\sE_D$ is $\mu$--stable, $\sE|_{D_z}$ is $\mu$--stable as well.
  Denote by $H_{D_z}$ the $\P$HYM metric on $\sE_{D_z}$ inducing the same metric on $\det(\sE|_{D_z})$ as $H_0|_{D_z}$.

  Using \autoref{Thm_MControlsS}, \autoref{Prop_MAdditivity}, $\log (H_{D_z}^{-1}H_0|_{D_z}) \in C^\infty_\delta(V,i\su(E,H_0))$, and \autoref{Prop_MUpperBound} we have
  \begin{align*}
    \Abs{s}_{L^2(D_z)} - 1
    &\lesssim
      \sM(H_{D_z},H_0e^s|_{D_z}) \\
    &=
      \sM(H_0|_{D_z},H_0e^s|_{D_z}) + \sM(H_{D_z},H_0|_{D_z}) \\
    &=
      \sM(H_0|_{D_z},H_0e^s|_{D_z}) + O(e^{-\delta L_0}) \\
    &\lesssim
      \int_{D_z} \abs{s}\abs{K_{H_0e^s|_{D_z}}} + e^{-\delta L_0}.
  \end{align*}
  This implies the asserted inequality.

  Comparing the lower bounds from \autoref{Step_IntegralLowerBound} with the upper bounds obtained by integrating \autoref{Step_SlicewiseUpperBounds} completes the proof of \autoref{Step_SlicewiseKLowerBound}.

  \begin{step}
    We have
    \begin{equation*}
      \Abs{K_{H_0e^s|_{D_z}}}_{L^2(V\setminus V_{L_0})}^2
      \lesssim e^{-\delta L_0} + \Abs{F_{H_0}^{\circ}}_{L^2(V_{L_0})}^2.
    \end{equation*}
    Here $F_{H_0}^{\circ}$ denotes the curvature of the $\PU(r)$--connection induced by $H_0$.
  \end{step}

  Once this is proved, the desired control on $M$ follows and the proof of \autoref{Prop_APrioriCkEstimate} will be complete.

  \begin{substep}
    We have
    \begin{equation*}
      \Abs{K_{H_0e^s|_{D_z}}}_{L^2(V\setminus V_{L_0})}^2
      \lesssim
        \int_V \abs{F_{H_0e^s}^{\circ}}^2 - \abs{F_{H_0}^{\circ}}^2
        + ce^{-\delta L_0}
        + \Abs{F_{H_0}^{\circ}}_{L^2(V_{L_0})}^2.
    \end{equation*}
  \end{substep}

  If $H$ is a Hermitian metric on a holomorphic bundle $\sE$ over an $n$--dimensional Kähler manifold $X$ with Kähler form $\omega$, then
  \begin{equation}
    \label{Eq_TopologicalEnergyBound}
    q_4(H) \wedge \omega^{n-2}
      = c\(\abs{F_{H}^{\circ}}^2 - \abs{K_{H}}^2\)\vol
  \end{equation}
  with
  \begin{equation*}
    q_4(H) :=
    2c_2(H) - \frac{r-1}{r}c_1(H)^2
  \end{equation*}
  and $c_k$ denoting the $k$--th Chern form associated with $H$.

  If $X$ is compact, then the integral of the left-hand side of \eqref{Eq_TopologicalEnergyBound} depends only $\sE$;
  hence,
  \begin{equation*}
    \int_{D_z} \abs{K_{H_0e^s|_{D_z}}}^2
    = \int_{D_z} \abs{K_{H_0|_{D_z}}}^2
    + \int_{D_z} \abs{F_{H_0e^s|_{D_z}}^{\circ}}^2 - \abs{F_{H_0|_{D_z}}^{\circ}}^2.
  \end{equation*}
  Since
  \begin{align*}
    \abs{F_{H_0} - F_{H_0|_{D_z}}}
    \lesssim
      e^{-\delta L} \qandq
    \abs{K_{H_0|_{D_z}}}
    \lesssim
      e^{-\delta L},
  \end{align*}
  it follows that
  \begin{align*}
    \int_{D_z} \abs{K_{H_0e^s|_{D_z}}}^2
    &\lesssim
      \int_{D_z} \abs{F_{H_0e^s|_{D_z}}^{\circ}}^2 - \abs{F_{H_0|_{D_z}}^{\circ}}^2 + e^{-\delta L} \\
    &\lesssim
      \int_{D_z} \abs{F_{H_0e^s}^{\circ}}^2 - \abs{F_{H_0}^{\circ}}^2 + e^{-\delta L}.
  \end{align*}

  \begin{substep}
    We have
    \begin{equation*}
      \int_V \abs{F_{H_0e^s}^{\circ}}^2 - \abs{F_{H_0}^{\circ}}^2
      \leq 0.
    \end{equation*}    
  \end{substep}

  Since $s \in C^\infty_\delta(V,i\su(E,H_0))$, we have
  \begin{equation*}
    \int_V \(q_4(H_0e^s) - q_4(H_0)\)\wedge \omega^{n-2} = 0.
  \end{equation*}
  Using \eqref{Eq_TopologicalEnergyBound}, we obtain
  \begin{equation*}
    \int_V \abs{F_{H_0e^s}^{\circ}}^2 - \abs{F_{H_0}^{\circ}}^2
    =
      \int_V \abs{K_{H_0e^s}}^2 - \abs{K_{H_0}}^2.
  \end{equation*}
  To see that the right-hand side is non-positive, we use \eqref{Eq_DeltaS2} and $\fL(s,t) = 0$ to derive
  \begin{equation*}
    \int_V \abs{K_{H_0e^s}}^2
    = \int_V t^2\abs{s}^2
    \leq \int_V t\abs{K_{H_0}}\abs{s}
    \leq \int_V \frac12\abs{K_{H_0}}^2+  \frac12 \abs{K_{H_0e^s}}^2.
    \qedhere
  \end{equation*}
\end{proof}

\subsection{Decay estimate}
\label{Sec_Decay}

\begin{proof}[Proof of \autoref{Prop_APrioriEstimate}]
  To complete the proof we need to establish quantitative exponential decay bounds for $s$ using the a priori estimate in \autoref{Prop_APrioriCkEstimate} and the qualitative information that $s \in C^\infty_\delta(V,i\su(E,H_0))$.

  Fix $L_0 \gg 1$ as in the proof of \autoref{Prop_APrioriCkEstimate}.

  \setcounter{step}{0}
  \begin{step}
    We have
    \begin{equation*}
      \int_{V\setminus V_{L_0}} \abs{\nabla_{H_0}s}^2 \leq c.
    \end{equation*}
  \end{step}

  From \autoref{Prop_LaplacianIdentity} and $\fL(s,t) = 0$ it follows that
  \begin{equation*}
    \Delta \abs{s}^2 + 2\abs{\upsilon(-s)\nabla_{H_0}s}^2
    \leq
      -4\inner{K_{H_0}}{s}.
  \end{equation*}
  Since
  \begin{equation*}
   \upsilon(-s)
    =
      \sqrt{\frac{1-e^{-\ad_s}}{\ad_s}}
    \qandq
    \sqrt{\frac{1-e^{-x}}{x}}
    \gtrsim
      \frac{1}{\sqrt{1+\abs{x}}},
  \end{equation*}
  it follows that
  \begin{equation}
    \label{Eq_NablaS}
      \abs{\nabla_{H_0} s}^2
      \lesssim
        (1+\Abs{s}_{L^\infty}) \(\abs{K_{H_0}}\abs{s} - \Delta \abs{s}^2\).
  \end{equation}
  Integrating this over $V$ and using \eqref{Eq_KH0Decay} as well as \autoref{Prop_APrioriCkEstimate} yields the asserted estimate.

  \begin{step}
    \label{Step_L2EpsilonDecay}
    For some $\epsilon > 0$ and all $L \geq L_0$, we have
    \begin{equation*}
      \int_{V\setminus V_L} \abs{s}^2
      \lesssim e^{-2\epsilon L}
      \qandq
      \int_{V\setminus V_L} \abs{\nabla_{H_0}s}^2
      \lesssim e^{-2\epsilon L}.
    \end{equation*}
  \end{step}

  Since $\sE_D$ is simple, for all $\tilde s \in \Gamma(D,\sEnd_0(\sE_D))$ we have
  \begin{equation*}
    \int_D \abs{\tilde s}^2
    \lesssim \int_D \abs{\delbar_D \tilde s}^2
    \lesssim \int_D \abs{\nabla_{H_D} \tilde s}^2.
  \end{equation*}
  Because $L_0 \gg 1$, this implies that
  \begin{equation}
    \label{Eq_SimplePoincare}
    \int_{\del V_L} \abs{s}^2
    \lesssim \int_{\del V_L} \abs{\nabla_{H_0} s}^2
  \end{equation}
  for $L \geq L_0$.
  Therefore, it suffices to prove the second inequality.
  
  Integrating \eqref{Eq_NablaS} over $V\setminus V_L$ and using \eqref{Eq_SimplePoincare} yields
  \begin{align*}
    \int_{V\setminus V_L} \abs{\nabla_{H_0}s}^2
    &\lesssim
      e^{-\delta L}
      + \int_{\del V_L} \abs{\nabla_{H_0}s}\abs{s} \\
    &\lesssim 
      e^{-\delta L}
      + \int_{\del V_L} \abs{\nabla_{H_0}s}^2.
  \end{align*}
  The assertion now follows from \autoref{Prop_AbstractDecay},
  which will be proved at the end of this section.

  \begin{step}
    \label{Step_CkEpsilonDecay}
    With $\epsilon > 0$ as above
    \begin{equation*}
      \Abs{s}_{C^{k,\alpha}_\epsilon} \leq c_{k,\alpha}.
    \end{equation*}
  \end{step}

  As in the proof of \autoref{Prop_Regularity}, we can write $\fL(s,t) = 0$ in the form
  \begin{equation}
    \label{Eq_SchematicPDE}
    \(\frac12\nabla_{H_0}^*\nabla_{H_0} + t\) s
    + B(\nabla_{H_0} s\otimes \nabla_{H_0} s)
    = \fe,
  \end{equation}
  where $B$ is linear with coefficients depending on $s$, and by \eqref{Eq_KH0Decay}
  \begin{equation*}
    \Abs{\fe}_{C^{k,\alpha}_\delta} \leq c_{k,\alpha}.
  \end{equation*}
  Using standard interior estimates the assertion follows from \autoref{Prop_APrioriCkEstimate} and \autoref{Step_L2EpsilonDecay}.

  \bigbreak
  \begin{step}
    We prove the proposition.
  \end{step}

  Since
  \begin{equation*}
    \Abs{\nabla_{H_0} s \otimes \nabla_{H_0} s}_{C^{k,\alpha}_{2\epsilon}}
    \lesssim
      \Abs{\nabla_{H_0}s}_{C^{k,\alpha}_\epsilon}^2,
  \end{equation*}
  we note that
  \begin{equation*}
    \Abs*{\frac12\nabla_{H_0}^*\nabla_{H_0}s+ts}_{C^{k,\alpha}_{\epsilon'}}
    \leq
      c_{k,\alpha}.
  \end{equation*}
  with $\epsilon' := \min\set{2\epsilon,\delta}$.
  From \autoref{Prop_TraceFreeLaplacian} it follows that
  \begin{equation*}
    \Abs{s}_{C^{k,\alpha}_{\epsilon'}}
    \leq
      c_{k,\alpha}.
  \end{equation*}
  Repeating this argument a finite number of times we finally arrive at $\epsilon' = \delta$.
\end{proof}

\begin{prop}
  \label{Prop_AbstractDecay}
  If $f \co [0,\infty) \to [0,\infty)$ satisfies
  \begin{equation*}
    f(L) \leq Ae^{-\delta L} - B f'(L)
  \end{equation*}
  with $A,B > 0$, then 
  \begin{equation*}
    f(L) \leq (2A+f(0))e^{-\epsilon L}
  \end{equation*}
  with $\epsilon := \min\set{\delta,1/2B}$.
\end{prop}

\begin{proof}
  The function $g\co [0,\infty) \to \R$ defined by
  \begin{equation*}
    g(L) := f(L) - (2A+f(0))e^{-\epsilon L}
  \end{equation*}
  satisfies $g(0) = -2A \leq 0$ and $g'(L) \leq -g(L)/B$.
  It follows that $g \leq 0$, which proves the proposition.
\end{proof}


\section{The Bando--Siu continuity method}
\label{Sec_BSContinuityMethod}

To prove \autoref{Thm_ACylDUY} for reflexive sheaves $\sE$ we use a regularization scheme based on ideas of Bando and Siu \cite{Bando1994}.
We construct a one-parameter family of ACyl Kähler manifolds $\set{\tilde V_\epsilon : \epsilon \in (0,1] }$ whose underlying complex manifold $\tilde V$ is obtained by blowing up $S := \sing(\sE)$.
As $\epsilon$ tends to zero, the exceptional divisor shrinks and $\tilde V_\epsilon$ resembles $V$ more and more closely.
$\tilde V$ carries a holomorphic vector bundle $\tilde \sE$, which agrees with $\sE$ outside $S$, and to which \autoref{Thm_ACylDUY} can be applied to construct a $\P$HYM metric $\tilde H_\epsilon$.
The desired $\P$HYM metric on $\sE$ will be constructed by taking the limit as $\epsilon$ tends to zero.

\begin{prop}
  \label{Prop_BlowUpSingularSet}
  There is a complex manifold $\tilde V$, a holomorphic map $\pi\co \tilde V \to V$ which induces a biholomorphic map to $V\setminus S$, and a holomorphic vector bundle $\tilde \sE$ over $\tilde V$ such that
  \begin{equation*}
    \tilde\sE|_{\tilde V \setminus \pi^{-1}(S)} \iso \pi^*(\sE|_{V\setminus S}).
  \end{equation*}
  Moreover, there exists a one-parameter family of Kähler metrics $\set{ g_\epsilon : \epsilon \in (0,1] }$ on $\tilde V$ such that:
  \begin{itemize}
  \item 
    on $\pi^{-1}(V\setminus B_{\sqrt \epsilon}(S))$
    we have $g_\epsilon = \pi^*g$, and
  \item
    for $L \geq L_0$, the Neumann--Poincaré constant of $(\pi^{-1}(V_L),g_\epsilon)$ is bounded above by a constant independent of $\epsilon$.
    Here $L_0$ is as in \autoref{Def_ATIReflexiveSheaf}.
  \end{itemize}
\end{prop}

\bigbreak
\begin{proof}
  The proof has three steps.

  \setcounter{step}{0}
  \begin{step}
    Construction of $\tilde V$ and $\tilde \sE$.
  \end{step}

  We follow the method of \citet[p.~46]{Bando1994}, see also \cite[Section 4.1]{Sibley2015}.

  Since $\sE^*$ is coherent,
  there exists a locally free sheaf $\sF$ and a surjective morphism $\sF^* \to \sE^* \to 0$.
  Since $\sE$ is reflexive, by dualising, we get $0 \to \sE \to \sF$.
  This defines a rational section $\phi_\sE \co V \dashrightarrow \Gr_r(\sF)$, with locus of indeterminacy $S$.
  By a result of \citet[Part I, Chapter 0, Section 5]{Hironaka1964}, there exists a holomorphic map $\pi \co \tilde V \to V$, which is biholomorphic outside $S$ and equivalent to a sequence of blow-ups along smooth submanifolds (of codimension at least three), such that  $\phi_\sE\circ \pi$ extends to a section $\tilde V \to \Gr_r(\pi^*\sF)$.
  This section defines the desired holomorphic vector bundle $\tilde \sE$ over $\tilde V$.

  \begin{step}
    \label{Step_ModelMetric}
    The model metric.
  \end{step}

  The Kähler form
  \begin{equation*}
    \tilde \omega_\epsilon
    = i\del\delbar\(\frac12\abs{z}^2 + \frac{\epsilon^2}{2\pi}\log\abs{z}^2\)
  \end{equation*}
  on $\C^n \setminus \set{0}$ uniquely extends to a Kähler form on $\Bl_0 \C^n$ which induces the $\epsilon^2$--times the Fubini--Study form $\omega_{FS}$ on the exceptional divisor $\P^{n-1}$.
  More precisely, if we denote by $r$ the radial coordinate, by $\theta$ the $1$--form arising from the $S^1$--action and by $\varpi\co \C^n \setminus \set{0} \to \P^{n-1}$ the projection, then
  \begin{equation*}
    \tilde\omega_\epsilon = (\epsilon^2+r^2)\varpi^*\omega_{FS} + r\rd r\wedge \theta.
  \end{equation*}

  Fix a smooth function $\chi \co [0,\infty) \to [0,1]$ which is equal to one on $[0,1]$ and vanishes outside $[0,2]$.
  For $0 < \epsilon \ll 1$, set $\chi_\epsilon := \chi(\cdot/2\sqrt\epsilon)$ and define a Kähler form on $\Bl_0 \C^n$ by
  \begin{equation*}
     \omega_\epsilon := i\del\delbar\(\frac12\abs{z}^2 + \chi_\epsilon(\abs{z})\cdot \frac{\epsilon^2}{2\pi} \log \abs{z}^2\).
  \end{equation*}
  This agrees with $\tilde\omega_\epsilon$ on $B_{\sqrt \epsilon/2}$, with $\omega_0$ on $\C^n{\setminus} B_{\sqrt\epsilon}(0)$ and satisfies 
  \begin{equation*}
    \abs{\omega_\epsilon-\omega_0} \lesssim \epsilon \abs{\log \epsilon}
  \end{equation*}
  on $B_{\sqrt\epsilon}(0) {\setminus} B_{\sqrt\epsilon/2}(0)$.
  Moreover, we have
  \begin{equation*}
    \frac{\omega_\epsilon^n}{\omega^n} \asymp 1+(\epsilon/r)^{2n-2}.
  \end{equation*}

  \bigbreak
  \begin{step}
    Construction of $g_\epsilon$.
  \end{step}

  $\tilde V$ is constructed by a sequence of blow-ups along smooth submanifolds.
  In fact, by induction we can assume that there is just one blow-up, say, along $C \subset V$.
  Denote by $\rho \co V \to [0,\infty)$ the distance to $C$.
  For $0 < \epsilon \leq \epsilon_0$,
  \begin{equation*}
    \omega_\epsilon := \pi^*\omega + i\del\delbar \(\chi_\epsilon \circ \rho \cdot \frac{\epsilon^2}{2\pi}  \log \rho^2\)
  \end{equation*}
  defines a Kähler form on $\tilde V$ whose restriction to $\pi^{-1}(V\setminus B_\epsilon(S))$ agrees with $\pi^*\omega$.
  We extend the resulting family of Kähler metrics to be constant for $\epsilon \in [\epsilon_0,1]$.

  \begin{step}
    Estimate of the Neumann--Poincaré constant.
  \end{step}

  Fix $L \geq L_0$.
  We use the discretization method of \citet[Section 3.1]{Grigoryan2004} to estimate the Neumann--Poincaré constant of $(\pi^{-1}(V_L),g_\epsilon)$.
  Fix $0 < \sigma \ll 1$.
  Pick a maximal set of points $\set{ x_j : j \in J } \subset V_{L-1/2}$ of distance at least $\sigma$ from each other.
  Set
  \begin{gather*}
    A_0 := V_{L}\setminus V_{L-1/2}, \quad
    A_0^* = A_0^{\#} := V_{L}\setminus V_{L-1},\\
    A_j := \pi^{-1}(B_\sigma(x_j)), \quad
    A_j^* = \pi^{-1}(B_{4\sigma}(x_j)) \qandq
    A_j^{\#} := \pi^{-1}(B_{8\sigma}(x_j)).
  \end{gather*}
  Set $I := J \sqcup \set{0}$.
  $\sA := \set*{ (A_i,A^*_i,A^{\#}_i) : i \in I }$ is a \defined{good covering} of $V_L$ in $V_L$ in the sense of \citet[Definition 3.1]{Grigoryan2004}.
  This means that, for all $i \in I$, $A_i \subset A_i^* \subset A_i^{\#}$ and for some constants $Q_1, Q_2$ the following hold: 
  \begin{itemize}
  \item
    We have $V_L \subset \bigcup_{i\in I} A_i$ and $\bigcup_{i\in I} A_i^{\#} \subset V_L$.
  \item
    For each $i \in I$, $\abs{\set{ j \in I : A_i^{\#} \cap A_j^{\#} \neq \emptyset }} \leq Q_1$.
  \item
    If $d(A_i,A_j) = 0$, then there is a $k = k(i,j) \in I$ such that $A_i \cup A_j \subset A_k^*$.
    Moreover, $\vol(A_k^*) \leq Q_2 \min\set{\vol(A_i),\vol(A_j)}$.
  \end{itemize}

  According to \cite[Theorem 3.7]{Grigoryan2004} the Neumann--Poincaré constant of $V_L$ can be estimated above by $Q_1 \Lambda_c(2 + Q_1^2Q_2\Lambda_d)$.
  Here the \defined{continuous Poincaré constant} $\Lambda_c$ and the \defined{discrete Poincaré constant} $\Lambda_d$ \cite[Definition 3.4 and Definition 3.6]{Grigoryan2004} are the smallest constants such that,
  \begin{equation}
    \label{Eq_LocalWeakPoincare}
    \int_{A_i} \abs{f-\bar f_{A_i}}^2 
      \leq 
      \Lambda_c \int_{A_i^*} \abs{\nabla f}^2
    \qandq
    \int_{A_i^*} \abs{f-\bar f_{A_i^*}}^2
      \leq
      \Lambda_c \int_{A_i^{\#}} \abs{\nabla f}^2
  \end{equation}
  and
  \begin{equation*}
    \sum_{i\in I} \abs{f(i) - \bar f}^2 m(i)
      \leq \Lambda_d \sE(f,f).
  \end{equation*}
  Here
  \begin{gather*}
    m(i) = \vol(A_i), \quad
    \bar f := \frac{\sum_{i \in I} f(i)m(i)}{\sum_{i\in I} m(i)} \qand \\
    \sE(f,f) := \frac12\sum_{(i,j) \in I\times I} \abs{f(i)-f(j)}^2 m(i,j).
  \end{gather*}
  with
  \begin{equation*}
    m(i,j) :=
    \begin{cases}
      \max\set{m(i),m(j)} & \text{if } d(A_i,A_j) = 0 \\
      0 & \text{otherwise}.
    \end{cases}
  \end{equation*}
  
  While the measures of $A_i$, $A_i^*$, and $A_i^{\#}$ are dependent of $\epsilon$, they are uniformly comparable.
  Consequently, the constants $Q_1$ and $Q_2$ and discrete Poincaré constant $\Lambda_d$ can be bounded independent of $\epsilon$.
  Thus it remains to show that $\Lambda_c$ can be bounded independent of $\epsilon$;
  that is, we can find a constant such that \eqref{Eq_LocalWeakPoincare} holds for all $i \in I$ and $\epsilon \in (0,1]$.
  For $i = 0$, \eqref{Eq_LocalWeakPoincare} is obvious.
  For $i \in J$, such estimates follow from scaling considerations and uniform weak Poincaré inequalities
  \begin{equation*}
    \int_{B_r(x)} \abs{f-\bar f_{B_r(x)}}^2
    \leq
      cr^2 \int_{B_{2r}(x)} \abs{\nabla f}^2
  \end{equation*}
  (with $c > 0$ independent of $x$ and $r$)  for certain model spaces, for example, $\Bl_0\C^k \times \C^{n-k}$ equipped with the Kähler metric induced by $i\del\delbar\(\frac12\abs{z}^2 + \frac{1}{2\pi}\log\abs{z}^2 + \frac12\abs{w}^2\)$.
  The existence of these uniform Poincaré constants in turn can also be  established using the discretization method as follows.
  We can assume that $r \gg 1$.
  Denote by $\pi \co \Bl_0\C^k \times \C^{n-k} \to \C^n$ the projection.
  For $i \in \Z^{2n} \subset \C^n$, set 
  \begin{equation*}
    A_i := \pi^{-1}(B_1(i)), \quad
    A_i^* := \pi^{-1}(B_4(i)) \qandq
    A_i^{\#} := \pi^{-1}(B_8(i)).
  \end{equation*}
  If we set $I_{x,r} := \set{ i \in \Z^{2n} \cap \pi(B_r(x)) }$, then $\sA_{x,r} := \set{ (A_i,A^*_i,A^{\#}_i) : i \in I_{x,r} }$ is a good covering of $B_r(x)$ in $B_{2r}(x)$; moreover, the constants $Q_1$ and $Q_2$ as well as the continuous Poincaré constant $\Lambda_c$ of $\sA_{x,r}$ can be bounded independent of $x$ and $r$.
  The discrete Poincaré constant of $\sA_{x,r}$ can be bounded by a constant times $r^2$;
  see, e.g., \cite[Section 3.4]{Berestycki2014}.
  \cite[Theorem 3.7]{Grigoryan2004} thus establishes the desired uniform weak Poincaré inequalities.
\end{proof}

We denote $\tilde V$ equipped with the metric $g_\epsilon$ by $\tilde V_\epsilon$.
Given a subset $U \subset V$, we set $\tilde U := \pi^{-1}(U)$.

Using \autoref{Thm_ACylDUY} for holomorphic vector bundles, for each $\epsilon \in (0,1]$, we construct a $\P$HYM metric $\tilde H_\epsilon$ on $\tilde\sE$ over $\tilde V_\epsilon$.
We can assume that the metric on $\det\tilde\sE$ induced by $\tilde H_\epsilon$ agrees with a fixed asymptotically translation-invariant metric $\tilde h$ which does not depend on $\epsilon$.
Define $\tilde {s}_\epsilon \in C^\infty_\delta(\tilde V_\epsilon,i\su(E,\tilde H_1))$ by
\begin{equation*}
  \tilde {s}_\epsilon := \log \tilde H_1^{-1}\tilde H_\epsilon.
\end{equation*}
The $\P$HYM metric $H$ on $\sE$, whose existence was asserted in \autoref{Thm_ACylDUY}, 
can be constructed using the following proposition and Arzelà--Ascoli
by taking the limit of the metrics $\tilde H_\epsilon$ over $V \setminus U = \tilde V_\epsilon \setminus \tilde U$ as $\epsilon$ tends to zero.
Here $U$ is an arbitrary neighbourhood of $S \subset V$.

\begin{prop}
  \label{Prop_BSAPrioriEstimate}
  For all $\epsilon \in (0,1]$, we have
  \begin{equation*}
    \Abs{\tilde {s}_\epsilon}_{C^k_\delta(\tilde V_\epsilon \setminus \tilde U)}
    \leq c_{k,U}.
  \end{equation*}
\end{prop}

\begin{proof}
  Set
  \begin{equation*}
    K_\epsilon
    := 
    i\Lambda_\epsilon F_{{\tilde H_1}} - \frac{\tr(i\Lambda_\epsilon F_{{\tilde H_1}})}{\rk \tilde\sE} \cdot \id_{\tilde\sE},
  \end{equation*}
  and let $f_\epsilon \in C^0_\delta(\tilde V_\epsilon)$ and $A_\epsilon > 0$ be the unique solution to
  \begin{equation*}
    \Delta_{\epsilon} (f_\epsilon-A_\epsilon\ell) = 4\abs{K_\epsilon}.
  \end{equation*}
  Here $\Lambda_\epsilon$ and $\Delta_\epsilon$ denote the dual Lefschetz operator and the Laplace operator on $\tilde V_\epsilon$ respectively.

  If we can prove that
  \begin{equation*}
    \Abs{f_\epsilon}_{L^\infty(\tilde V_\epsilon \setminus \tilde U)} \leq c_U, \quad
    A_\epsilon \leq c \qandq
    \Abs{F_{\tilde H_1}}_{L^2(\tilde V_{\epsilon,L_0})} \leq c,
  \end{equation*}
  then the argument in \autoref{Sec_APrioriEstimate} will yield the asserted bounds on $\tilde {s}_\epsilon$.

  The proof of the above bounds on $f_\epsilon$, $A_\epsilon$ and $F_{\tilde H_1}$ proceeds in four steps.

  \setcounter{step}{0}
  \begin{step}
    \label{Step_FL2KCKBound}
    We have
    \begin{equation*}
      \Abs{F_{\tilde H_1}}_{L^2(\tilde V_{\epsilon,L_0})} \leq c
      \qandq
      \Abs{K_\epsilon}_{C^k_\delta(V\setminus V_{L_0})} \leq c_k;
    \end{equation*}
    in particular, $A_\epsilon \leq c$.
  \end{step}

  By scaling considerations, we have
  \begin{equation*}
    \abs{F_{\tilde H_1}}_{g_\epsilon}^2 \vol_{g_\epsilon}
    \lesssim \(\frac{\rho^2+\epsilon^2}{\rho^2+1}\)^{\codim(S)-3} \abs{F_{\tilde H_1}}^2 \vol_{g_1}.
  \end{equation*}
  Since $\codim S \geq 3$, this implies the asserted $L^2$--bound.
  The second inequality is a consequence of the fact that $g_\epsilon$ and, thus, $K_\epsilon$ does not depend on $\epsilon$ on $V\setminus V_{L_0}$.
  Both estimates together yield $A_\epsilon \lesssim \Abs{K_\epsilon}_{L^1(\tilde V_\epsilon)} \leq c$.

  \begin{step}
    There is a constant $\bar f_\epsilon$, such that on $V \setminus V_{L_0}$ we have
    \begin{equation*}
      \Abs{e^{-\frac{\delta\ell}{2}}(f_\epsilon-\bar f_\epsilon)}_{L^2(\tilde V_\epsilon)}
      \leq
        c
      \qandq
      \Abs{\nabla_\epsilon f_\epsilon}_{L^2(\tilde V_\epsilon)}^2
      \leq
        c.
    \end{equation*}
  \end{step}

  From \autoref{Prop_BlowUpSingularSet} it follows that the weighted Neumann--Poincaré inequality \cite[Theorem 4.18]{Haskins2012} holds for $\sigma = 1$ and $\mu = \frac{\delta}{2}$ with a constant $c > 0$ independent of $\epsilon$;
  hence, for some constant $\bar f_\epsilon$
  \begin{equation*}
    \Abs{e^{-\frac{\delta\ell}{2}}(f_\epsilon-\bar f_\epsilon)}_{L^2(\tilde V_\epsilon)}^2
    \lesssim
    \Abs{\nabla_\epsilon f_\epsilon}_{L^2(\tilde V_\epsilon)}^2.
  \end{equation*}

  Using the previous step, we have
  \begin{align*}
    \Abs{\nabla_\epsilon f_\epsilon}_{L^2(\tilde V_\epsilon)}^2
    &=
      \int_{\tilde V_\epsilon} \inner{\Delta_\epsilon (f_\epsilon-\bar f_\epsilon)}{f_\epsilon-\bar f_\epsilon} \\
    &\leq
      \Abs{e^{\frac{\delta\ell}{2}}(K_\epsilon+A_\epsilon\Delta_\epsilon\ell)}_{L^2(\tilde V_\epsilon)}
      \cdot
      \Abs{e^{-\frac{\delta\ell}{2}}(f_\epsilon-\bar f_\epsilon)}_{L^2(\tilde V_\epsilon)} \\
    &\lesssim
      \Abs{e^{-\frac{\delta\ell}{2}}(f_\epsilon-\bar f_\epsilon)}_{L^2(\tilde V_\epsilon)}.
  \end{align*}
  Combined with the above this yields
  \begin{equation*}
    \Abs{e^{-\frac{\delta\ell}{2}}(f_\epsilon-\bar f_\epsilon)}_{L^2(\tilde V_\epsilon)}
      \leq c.
  \end{equation*}
  This in turn implies the second of the asserted inequalities.

  \begin{step}
    We have
    \begin{equation*}
      \Abs{f_\epsilon}_{L^\infty(\tilde V_\epsilon \setminus U)} \leq c_U.
    \end{equation*}
  \end{step}

  Define $F \co [L_0,\infty) \to [0,\infty)$ by
  \begin{equation*}
    F(L) := \int_{V\setminus V_{L_0}} \abs{\nabla_\epsilon f_\epsilon}^2.
  \end{equation*}
  By the previous step, we have
  \begin{equation*}
    F(L) \leq c.
  \end{equation*}

  Setting $\bar f_{\epsilon,L} := \fint_{\del V_L} f_\epsilon$,
  we have
  \begin{equation*}
    \int_{\del V_L} \abs{f_\epsilon-\bar f_{\epsilon,L}}
    \leq \int_{\del V_L}\abs{f_\epsilon-\bar f_\epsilon}.
  \end{equation*}
  By integration by parts, the Neumann--Poincaré inequality on $\del V_L$ and the previous step, we have
  \begin{align*}
    F(L)
    &\leq
      \int_{V\setminus V_L} \abs{K_\epsilon+A_\epsilon\Delta\ell}\abs{f_\epsilon-\bar f_{\epsilon,L}}
      + \int_{\del V_L} \abs{\nabla_\epsilon f}\abs{f_\epsilon-\bar f_{\epsilon,L}} \\
    &\lesssim
      \int_{V\setminus V_L} e^{-\delta\ell}\abs{f_\epsilon-\bar f_\epsilon}
      + \int_{\del V_L} \abs{\nabla_\epsilon f}\abs{f_\epsilon-\bar f_{\epsilon,L}} \\
    &\lesssim
      e^{-\frac{\delta L}{2}} - F'(L).
  \end{align*}
  It follows from \autoref{Prop_AbstractDecay} that $F(L) \lesssim e^{-2\gamma L}$ for some $\gamma > 0$.
  From interior estimates it follows that
  \begin{equation*}
    \abs{\nabla_\epsilon f_\epsilon} \lesssim e^{-\gamma\ell}
  \end{equation*}
  on $V\setminus V_{L_0}$ and
  \begin{equation*}
    \Abs{\nabla_\epsilon f_\epsilon}_{L^\infty(\tilde V_\epsilon\setminus U)} \leq c_U.
  \end{equation*}
  This implies the assertion by integrating back from the end of $V$.
\end{proof}

The $L^2$ curvature bound asserted in \autoref{Thm_ACylDUY} is a consequence of the following proposition.

\begin{prop}
  \label{Prop_L2CurvatureBound}
  For each $\epsilon \in (0,1]$, we have
  \begin{equation*}
    \Abs*{F_{\tilde H_\epsilon}}_{L^2(\tilde V_{\epsilon,L})}
      \lesssim L+1.
  \end{equation*}
\end{prop}

\begin{proof}
  Since $\tilde h$ is fixed, it suffices to estimate $F_{\tilde H_\epsilon}^{\circ}$, the curvature of the $\PU(r)$--connection induced by $\tilde H_\epsilon$.

  For each fixed $\epsilon \in (0,1]$, we have a bound of the desired form;
  however, it might a priori depend on $\epsilon$.
  To see that it does not, we use a topological argument.
  With $q_4$ as defined in \eqref{Eq_TopologicalEnergyBound} we have
  \begin{equation*}
    q_4(\tilde H_\epsilon) - q_4(\tilde H_1)
    = \rd \tau(\tilde {s}_\epsilon)
  \end{equation*}
  where $\tau$ is the transgression form associated with $q_4$ and can be bounded in terms of $\abs{\tilde{s}_\epsilon}$ and $\abs{\nabla_{\tilde H}\tilde{s}_\epsilon}$.
  Using \eqref{Eq_TopologicalEnergyBound} and $K_{\tilde H_\epsilon} = 0$, we derive
  \begin{align*}
    \int_{\tilde V_L} \abs*{F^{\circ}_{\tilde H_\epsilon}}^2 \vol_\epsilon
    & \lesssim
      \int_{\tilde V_L} q_4(\tilde H_\epsilon) \wedge \omega_\epsilon^{n-2} \\
    &=
      \int_{\tilde V_L} (q_4(\tilde H_1) + \rd \tau) \wedge \omega_\epsilon^{n-2} \\
    &\lesssim
      \int_{\tilde V_L} \abs*{F^{\circ}_{\tilde H_1}}_{g_\epsilon}^2 \vol_\epsilon + 1 \\
    &\lesssim
      \int_{\tilde V_L} \abs*{F^{\circ}_{\tilde H_1}}_{g_1}^2 \vol_1 + 1 \\
    &\lesssim
      L+1.
  \end{align*}
  Here the second term in the third step arises from Stokes' theorem and the fourth step uses the argument from \autoref{Step_FL2KCKBound} in the proof of \autoref{Prop_BSAPrioriEstimate}.
\end{proof}

This finishes the proof of \autoref{Thm_ACylDUY}.


\section{Uniqueness of  \texorpdfstring{$\P$HYM}{PHYM} metrics}
\label{Sec_Uniqueness}

We have the following basic uniqueness result for asymptotically translation-invariant $\P$HYM metrics.

\begin{prop}
  Let $\sE$ be a reflexive sheaf over $V$ asymptotic to $\sE_D$ and let $h$ be an asymptotically translation-invariant Hermitian metric on $\det E$.
  If $\sE_D$ is simple, then there exist at most one asymptotically translation-invariant $\P$HYM metric on $\sE$ inducing $h$.
\end{prop}

\begin{proof}
  If $H_0$ and $H$ were two asymptotically translation-invariant $\P$HYM metrics inducing $h$, then they must be asymptotic to the same $\P$HYM metric $H_D$ on $\sE_D$ (by uniqueness in the compact case).
  Then, for some $\delta > 0$,
  \begin{equation*}
    s := \log(H_0^{-1} H) \in C^\infty_\delta(V\setminus S,i\su(E,H_0)).
  \end{equation*}
  Moreover, by \cite[p.~13]{Siu1987},
  \begin{equation*}
    \Delta \log \tr e^s \leq 0
  \end{equation*}
  on $V\setminus S$.
  The argument in the proof of \cite[Theorem 2(a)]{Bando1994} shows that $\log \tr e^s \in W^{1,2}_\loc(V)$; hence, $\log \tr e^s$ is weakly subharmonic and thus $\log \tr e^s \leq \log \rk \sE$.
  However, because of the inequality of arithmetic and geometric means, $\log \tr e^s \geq \log \rk \sE$ with equality if and only if $s = 0$.
\end{proof}


\appendix
\section{Useful formulae for Chern connections}
\label{Sec_UsefulFormulae}

Let $\sE = (E,\delbar)$ be a rank $r$ holomorphic vector bundle.
Given a Hermitian metric $H$ on $\sE$, there exists a unique Hermitian covariant derivative $\nabla = \nabla_H$ on $E$ such that $\nabla_H^{0,1} = \delbar$.
The connection $A_H$ associated with $\nabla_H$ is called the \defined{Chern connection} induced by $H$.

Fix a Hermitian metric $H_0$ and $s \in i\fu(E,H_0)$.
Set
\begin{equation*}
  \tilde A_s := e^{s/2}_* A_H.
  \end{equation*}
Since $e^{s/2}_*H = H_0$, both $\tilde A_0 = A_{H_0}$ and $\tilde A_s$ are connections on the principal $\U(r)$--bundle $\U(E,H_0)$.
Set
\begin{equation*}
  \fK(s) := \Ad(e^{s/2})K_{H_0e^s}.
\end{equation*}
All of the following results can be found in \cite[Section 6.1]{Luebke2006}, in the setting of holomorphic principal bundles.
We summarise them here for the reader's convenience.

\begin{prop}
  \label{Prop_KH0ExpS}
  We have
  \begin{align*}
    \fK(s)
    &=
      (2\cosh(\ad_{s/2})-1) K_{H_0} \\ &\quad
      + \frac12\Theta(s) \nabla_{H_0}^*\nabla_{H_0} s \\ &\quad
      + \frac{i}2\Lambda\(\delbar \Upsilon(-s/2) \wedge \del_{H_0} s\)
      - \frac{i}2\Lambda\(\del_{H_0}\Upsilon(s/2) \wedge \delbar s\) \\ &\quad
      - \frac{i}4\Lambda\(
         \Upsilon(-s/2)\del_{H_0}s \wedge \Upsilon(s/2)\delbar s
         + \Upsilon(s/2)\delbar s\wedge \Upsilon(-s/2)\del_{H_0}s
        \)
  \end{align*}
  with $\Upsilon(s) \in \End(\gl(E))$ defined by
  \begin{equation}
    \label{Eq_Upsilon}
    \Upsilon(s)
    := \frac{e^{\ad_s}-1}{\ad_s}
  \end{equation}
  and $\Theta(s) \in \End(\gl(E))$ defined by
  \begin{equation*}
    \Theta(s)
    := \frac{\Upsilon(s/2)+\Upsilon(-s/2)}{2}.
  \end{equation*}
\end{prop}

\begin{remark}
  Since $\ad_s := [s,\cdot] \in \End(\gl(E))$ is self-adjoint with respect to $H_0$, so is $\Upsilon(s)$.
  Both $\cosh(\ad_{s/2})$ and $\Theta(s)$ preserve $\fu(E,H_0)$ because their power series expansions involve only even powers of $\ad_s$ and $\ad_s^2$ preserves $\fu(E,H_0)$.
  Also note that $\Theta(s)$ self-adjoint with respect to $H_0$ and its first eigenvalue is at least one.
\end{remark}

\begin{proof}[Proof of \autoref{Prop_KH0ExpS}]
  Since $\del_{H_0e^s} = \del_{H_0} + e^{-s}\del_{H_0}e^{s}$, we have
  \begin{align*}
    \del_{\tilde A_s} 
    &=
      e^{s/2}(\del_{H_0} + e^{-s}(\del_{H_0} e^s))e^{-s/2} \\
    &=
      \del_{H_0} + e^{s/2}\del_{H_0}e^{-s/2} + e^{-s/2}(\del_{H_0} e^s)e^{-s/2} \\
    &=
      \del_{H_0}
      +e^{-s/2}(\del_{H_0} e^{s/2})
  \end{align*}
  and
  \begin{align*}
    \delbar_{\tilde A_s} 
    =
      e^{s/2} \delbar e^{-s/2} 
    =
      \delbar+
      e^{s/2}(\delbar e^{-s/2})
      =
      \delbar -
      (\delbar e^{s/2})e^{-s/2}.
  \end{align*}
  Using
  \begin{equation*}
    \rd_x \exp (y)
    = (\Upsilon(x) y) e^{x}
    = e^{x} (\Upsilon(-x) y)
  \end{equation*}
  we obtain
  \begin{equation*}
    \tilde A_s
    =
      \tilde A_0 + \frac12\Upsilon(-s/2)\del_{H_0}s - \frac12\Upsilon(s/2)\delbar s
  \end{equation*}

  From this it follows that
  \begin{align*}
    F_{\tilde A_s}
    &=
      F_{H_0}
      + \frac12\Upsilon(-s/2) \delbar\del_{H_0} s
      - \frac12\Upsilon(s/2)\del_{H_0}\delbar s \\ &\quad
      + \frac12\delbar \Upsilon(-s/2) \wedge \del_{H_0} s
      - \frac12\del_{H_0}\Upsilon(s/2) \wedge \delbar s \\ &\quad
      - \frac14\(
         \Upsilon(-s/2)\del_{H_0}s \wedge \Upsilon(s/2)\delbar s
         + \Upsilon(s/2)\delbar s\wedge \Upsilon(-s/2)\del_{H_0}s
        \).
  \end{align*}
  Applying $i\Lambda$ and using the Kähler identities
  \begin{equation*}
    i[\Lambda,\delbar] = \del_{H_0}^*
    \qandq
    i[\Lambda,\del_{H_0}] = -\delbar^*
  \end{equation*}
  as well as
  \begin{equation*}
    \del_{H_0}^*\del_{H_0} = \frac12 \nabla_{H_0}^*\nabla_{H_0} - [K_{H_0},\cdot]
    \qandq
    \delbar^*\delbar = \frac12 \nabla_{H_0}^*\nabla_{H_0} + [K_{H_0},\cdot],
  \end{equation*}
  we obtain
  \begin{align*}
    e^{s/2}K_{H_0e^s}
    &=
      K_{H_0} + \frac12(\Upsilon(s/2)-\Upsilon(-s/2))\ad_s K_{H_0} \\ &\quad
      + \frac14(\Upsilon(s/2)+\Upsilon(-s/2)) \nabla_{H_0}^*\nabla_{H_0} s \\ &\quad
      + \frac{i}2\Lambda\(\delbar \Upsilon(-s/2) \wedge \del_{H_0} s\)
      - \frac{i}2\Lambda\(\del_{H_0}\Upsilon(s/2) \wedge \delbar s\) \\ &\quad\quad
      - \frac{i}4\Lambda\(
         \Upsilon(-s/2)\del_{H_0}s \wedge \Upsilon(s/2)\delbar s
         + \Upsilon(s/2)\delbar s\wedge \Upsilon(-s/2)\del_{H_0}s
        \).
  \end{align*}
  This implies the asserted identity.
\end{proof}

\begin{prop}
  \label{Prop_DerivativeKH0ExpS}
  We have
  \begin{align*}
    \rd_s \fK(\hat s)
    &=
      \frac12\nabla_{\tilde A_s}^*\nabla_{\tilde A_s} \Ad(e^{s/2})\Upsilon(-s)\hat s.
  \end{align*}
\end{prop}

\begin{proof}
  The asserted identity clearly holds at $s = 0$.
  To prove the general case, note that if $\sigma_t$ satisfies
  \begin{equation*}
    e^{s+t\hat s} = e^{s}\Ad(e^{-s/2})e^{\sigma_t},
  \end{equation*}
  then
  \begin{equation*}
    \left.\frac{\rd}{\rd t}\right|_{t=0}\sigma_t
    =
    \Ad(e^{s/2})\Upsilon(-s)\hat s.
    \qedhere
  \end{equation*}
\end{proof}

\begin{prop}
  \label{Prop_UpsilonHatSS}
  We have
  \begin{equation*}
    \inner{\Ad(e^{s/2})\Upsilon(-s)\hat s}{\hat s} \geq \abs{\hat s}^2.
  \end{equation*}
\end{prop}

\begin{proof}
  Since $\Ad(e^{s/2})\Upsilon(-s) = e^{\ad_{s/2}}\Upsilon(-s)$, this follows by observing that
  \begin{equation*}
    \frac{e^{x/2}(1-e^{-x})}{x}
    = \frac{\sinh(x/2)}{x/2} \geq 1
  \end{equation*}
  for all $x \in \R$.
\end{proof}

\begin{prop}
  \label{Prop_LaplacianIdentity}
  We have
  \begin{equation}
    \label{Eq_LaplacianIdentity}
    \inner{\fK(s)-K_{H_0}}{s}
    = \inner{i\Lambda \delbar (e^{-s} \del_{H_0} e^s)}{s}
    = \frac14 \Delta \abs{s}^2 + \frac12\abs{\upsilon(-s)\nabla_{H_0}s}^2
  \end{equation}
  where $\upsilon(s) \in \End(\gl(E))$ is defined by $\upsilon(s) := \sqrt{\Upsilon(s)}$.
\end{prop}

\begin{proof}
  We compute
  \begin{align*}
    \inner{i\Lambda\delbar (e^{-s} \del_{H_0} e^s)}{s}
    &=
      \inner{i\Lambda\delbar (\Upsilon(-s)\del_{H_0}s)}{s} \\
    &=
      i\Lambda\delbar\inner{\Upsilon(-s)\del_{H_0}s}{s} + i\Lambda\winner{\Upsilon(-s)\del_{H_0}s}{\del_{H_0}s} \\
    &=
      \del^*\inner{\del_{H_0}s}{\Upsilon(s)s} + \inner{\Upsilon(-s)\del_{H_0}s}{\del_{H_0}s} \\
    &=
      \del^*\inner{\del_{H_0}s}{s} + \abs*{\upsilon(-s)\del_{H_0}s}^2  \\
    &=
      \frac12 \del^*\del \abs{s}^2 + \abs{\upsilon(-s)\del_{H_0}s}^2.
    \qedhere
  \end{align*}
\end{proof}


\section{The Donaldson functional}
\label{Sec_DonaldsonFunctional}

Let $(X,g,I)$ be a compact Kähler manifold, let $\sE$ be a holomorphic vector bundle over $X$.
Given metric $H_0$ and $s \in C^\infty(X,i\su(\sE,H_0))$, the value of the Donaldson functional at $(H_0,H_0e^s)$ is
\begin{equation*}
  \sM(H_0,H_0e^s)
  := \int_0^1 \int_X \inner{s}{\Ad(e^{us/2})K_{H_0e^{u s}}} \,\rd u.
\end{equation*}
This functional was introduced in \cite[Section 1.2]{Donaldson1985} and \cite[§II]{Donaldson1986}.
We refrain from a lengthy discussion and only marshal the following three facts, which are used in \autoref{Sec_APrioriEstimate}.

\begin{prop}[{\cite[Proposition 5.1]{Simpson1988}}]
  \label{Prop_MAdditivity}
  We have
  \begin{equation*}
    \sM(H_0,H_2) = \sM(H_0,H_1) + \sM(H_1,H_2).
  \end{equation*}
\end{prop}

\begin{prop}
  \label{Prop_MUpperBound}
  We have
  $\sM(H_0,H_0e^s) \lesssim \int_X \abs{s}\abs{K_{H_0e^s}}$.
\end{prop}

\begin{proof}
  This holds because $m(u) := \sM(H_0,H_0e^{us})$ is convex \cite[Proof of Lemma 24]{Donaldson1986}, $m(0) = 0$ and $m'(1) \lesssim \int_X \abs{s}\abs{K_{H_0e^s}}$.
\end{proof}

\begin{theorem}[{\citet[Lemma 24]{Donaldson1986}; see also \cite[Proposition 5.3]{Simpson1988}}]
  \label{Thm_MControlsS}
  If $H_0$ is $\P$HYM, then
  \begin{equation*}
    \Abs{s}_{L^2}-1 \lesssim \sM(H_0,H_0e^s). 
  \end{equation*}
\end{theorem}


\section{Bando--Siu interior estimate}
\label{Sec_BSInteriorEstimate}

\begin{theorem}[{\citet[Proposition 1]{Bando1994}}]
  \label{Thm_BSInteriorEstimate}
  Let $(X,g,I)$ be a Kähler manifold of dimension $n$ with bounded geometry and let $\sE$ be a holomorphic vector bundle over $X$.
  If $H_0$ and $H$ are Hermitian metrics on $\sE$ and $s := \log(H_0^{-1}H) \in C^\infty(X,i\su(\sE,H_0))$,
  then
  \begin{align*}
    &r^{k+2-\frac{2n}{p}}\Abs{\nabla_{H_0}^{k+2} s}_{L^p(B_r(x))} \\
    &\quad\quad 
      \leq \epsilon_{k,p}\Big(
      \Abs{s}_{L^\infty(B_{2r}(x))} + \Abs{K_{H}}_{L^\infty(B_{2r}(x))} + r^{k-\frac{2n}{p}}\Abs{\nabla^k K_{H}}_{L^{p}(B_{2r}(x))} \\
    &\quad\quad\quad\quad\quad
      + \sum_{j=0}^k r^{2+i}\Abs{\nabla_{H_0}^i F_{H_0}}_{L^\infty(B_{2r}(x))}
      \Big)
  \end{align*}
  where $\epsilon_{k,p}$ is a smooth function which vanishes at the origin and depends only on $k \in \N$, $p \in (1,\infty)$, and the geometry of $X$.
\end{theorem}

It suffices to prove this in the case where $H_0$ is a flat metric on a trivial holomorphic bundle over $\bar B_2 \subset \C^n$.
The theorem is not a straight-forward consequence of standard bootstrapping techniques because we only have
\begin{equation*}
  \Delta s
  = A(K_H) + C(\nabla s \otimes \nabla s)
\end{equation*}
where $A$ and $C$ are linear with coefficients depending on $s$;
see \autoref{Prop_KH0ExpS}.
The usual Sobolev estimates will not suffice to prove \autoref{Thm_BSInteriorEstimate} without any control of $\nabla s$.
However, if we assume $C^{0,\beta}$ bounds on $\nabla s$ of the above form, then the usual method does give the desired estimates.
It is well known to analysts that for an equation of this form $C^{0,\beta}$ bounds on $\nabla s$ can be obtained provided a bound on the Morrey norm $\Abs{\nabla s}_{L^{2,2n-2+2\alpha}}$;
see \autoref{Def_Morrey}.
We give full details for this fact, which is completely general and has nothing to do with Hermitian Yang--Mills metrics, in \autoref{Sec_Hildebrandt}.
All of this being said, it thus suffices to prove the following proposition.

\bigbreak

\begin{prop}
  Denote by $H_0$ a flat Hermitian metric on the trivial holomorphic bundle of rank $r$ over $\bar B_2 \subset \C^n$.
  If $H = H_0e^s$ with $s \in C^\infty(\bar B_2,i\su(r))$, then
  \begin{align*}
    [s]_{C^{0,\alpha}(\bar B_1)}
    &\lesssim
      \Abs{\nabla s}_{L^{2,2n-2+2\alpha}(B_1)}\\
    &\leq
      \epsilon(\Abs{s}_{L^\infty(B_2)}+\Abs{K_H}_{L^\infty(B_2)})
  \end{align*}
  where $\alpha \in (0,1)$ depends on $\Abs{s}_{L^\infty(B_2)}$ in a monotonely decreasing way,
  and $\epsilon$ is a smooth function which vanishes at the origin.
\end{prop}

\begin{proof}
  For $x \in B_{1}$ define $f_x \co (0,1] \to [0,\infty)$ by
  \begin{equation*}
    f_x(r) := \int_{B_r(x)} G_x \abs{\nabla s}^2
  \end{equation*}
  with $G_x(\cdot) := \abs{\cdot-x}^{2-2n}$.
  We will show that
  \begin{equation*}
    f_x(r) \leq \epsilon r^{2\alpha}
  \end{equation*}
  with $\epsilon$ and $\alpha$ as in the proposition.
  This implies the asserted Morrey bound.

  In the following we fix $x \in B_1$ and $r \in (0,1/2]$ and omit writing the subscript $x$ to simplify notation.

  \setcounter{step}{0}
  \begin{step}
    We have $f(r) \leq \epsilon$.
  \end{step}

  Fix a smooth function $\chi \co [0,\infty) \to [0,1]$ which is equal to one on $[0,1]$ and vanishes outside $[0,2]$.
  Set $\chi_{r}(\cdot) := \chi(\abs{\cdot - x}/r)$.
  Using
  \begin{equation*}
    \abs{\nabla s}^2
    \lesssim
      \epsilon\cdot(1-\Delta\abs{s}^2),
  \end{equation*}
  which follows from \eqref{Eq_NablaS}, we compute
  \begin{align*}
    f(r) 
    &\leq
      \int_{B_{2r}(x)} \chi_{r} G \cdot \abs{\nabla s}^2 \\
    &\lesssim
      \epsilon\int_{B_{2r}(x)} \chi_{r} G \cdot (-\Delta\abs{s}^2)
      + \chi_{r} G \\
    &\lesssim
      \epsilon r^{-n} \int_{B_{2r}(x)\setminus B_{r}(x)} \abs{s}^2
      + \epsilon r^2 \\
    &\leq \epsilon.
  \end{align*}
  Here we used the convention of ``generic constants''; that is, $\epsilon$ is allowed to increase from one line to the next.

  \begin{step}
    We have $f(r) \leq \gamma f(2r) + \epsilon r^2$ for some constant $\gamma \in (0,1)$ depending on $\|s\|_{L^\infty(B_2)}$.
  \end{step}
  
  Set
  \begin{equation*}
    \bar s := \fint_{B_{2r}(x)\setminus B_r(x)} s \in i\su(r) \qandq \sigma := \log(e^{s}e^{-\bar s}).
  \end{equation*}
  Observe that
  \begin{equation*}
    \abs{\nabla s}^2
    \lesssim
    M\abs{\nabla\sigma}^2
    \qandq
    |\sigma|^2
    \lesssim
      M\abs{s-\bar s}^2
  \end{equation*}
  with $M>0$ some constant depending on $\Abs{s}_{L^\infty(B_2)}$ and $\Abs{K_H}_{L^\infty(B_2)}$ in a monotonely increasing way.
  Arguing as in the previous step we have
  \begin{equation*}
    \abs{\nabla \sigma}^2 
    \leq
      M\(-4\inner{K_{H}}{\sigma} - \Delta\abs{\sigma}^2\)
    \leq
      M(1 - \Delta\abs{\sigma}^2).
  \end{equation*}
  Using the above and Poincaré's inequality we have
  \begin{align*}
    \int_{B_r(x)} G\abs{\nabla s}^2
    &\lesssim
      M\int_{B_{2r}(x)} \chi_r G\cdot (-\Delta\abs{\sigma}^2)
      + \epsilon \chi_r G \\
    &\lesssim
      M\cdot r^{-2n} \int_{B_{2r}(x)\setminus B_r(x)} \abs{\sigma}^2
      + \epsilon r^2 \\
    &\lesssim
      M^2\cdot r^{-2n} \int_{B_{2r}(x)\setminus B_r(x)} \abs{s-\bar s}^2
      + \epsilon r^2 \\
    &\lesssim
      M^2 \cdot r^{2-2n} \int_{B_{2r}(x)\setminus B_r(x)} \abs{\nabla s}^2
      + \epsilon r^2 \\
    &\lesssim
      M^2\int_{B_{2r}(x)\setminus B_r(x)} G\abs{\nabla s}^2
      + \epsilon r^2.
  \end{align*}
  This gives the asserted inequality.

  \begin{step}
    We have $f(r) \leq \epsilon r^{2\alpha}$.
  \end{step}

  We can assume that $\gamma \geq 1/2$.
  Set $g(r) := f(r) + \frac{c}{4\gamma-1}\epsilon r^2$.
  By the second step
  \begin{equation*}
    g(r)
    \leq \gamma^k g(2^kr).
  \end{equation*}
  Setting $k := \log_2 \ceil{1/2r}$, we have $\gamma^k \lesssim r^{2\alpha}$ for some $\alpha \in (0,1)$ depending only on $\gamma$;
  hence, by the first step
  \begin{equation*}
    f(r) \leq \epsilon r^{2\alpha}.
    \qedhere
  \end{equation*}
\end{proof}


\section{Hildebrandt's \texorpdfstring{$C^{1,\beta}$}{C1beta} estimate}
\label{Sec_Hildebrandt}

The following result is well-known to analysts.
It can be traced back to Hildebrandt's work on harmonic maps \cite[Section 6]{Hildebrandt1985}.

\begin{prop}
  \label{Prop_HildebrandtAPrioriEstimate}
  Suppose $\alpha \in (0,1)$.
  Let $U$ be an open subset of $\R^n$ with smooth boundary and let $f \co \bar U \to \R^k$ be a solution of a partial differential equation of the form 
  \begin{equation}
    \label{Eq_PDE}
    \Delta f = A + B(\nabla f) + C(\nabla f \otimes \nabla f)
  \end{equation}
  where $A \in C^0(\bar U,\R^k)$,
  $B \in C^0(\bar U,\End(\R^k))$,
  and $C \in C^0(\bar U,\Hom(\R^k\otimes \R^k,\R^k))$.
  For each $V \subset\subset U$, we have
  \begin{equation*}
    \Abs{\nabla f}_{C^{0,\beta}(V)} \leq \epsilon\(\Abs{\nabla f}_{L^{n-2+2\alpha,2}(U)}\)
  \end{equation*}
  where $\epsilon$ is a smooth increasing function vanishing at the origin (depending on $A$, $B$, $C$, $U$ and $V$), and $\beta \in (0,1)$ depends only on $\alpha$.
\end{prop}

We will make heavy use of Morrey and Campanato spaces.
For the reader's convenience all necessary definitions and results are summarised in \autoref{Sec_MorreyCampanato}.

\begin{proof}
  Set $R := d(V,\del U)$.
  Define $\phi \co [0,R] \to [0,\infty)$ by
  \begin{equation*}
    \phi(r) := \sup \set*{ \int_{B_r(x)} \abs{\nabla f - \overline{\nabla f}_{x,r}}^2 : x \in V }.
  \end{equation*}
  By definition
  \begin{equation*}
    [\nabla f]_{\cL^{2,\lambda}(V)}
    \leq
    \sup \set*{ r^{-\lambda}\phi(r) : r > 0 }
    \leq
    [\nabla f]_{\cL^{2,\lambda}(U)}.
  \end{equation*}

  We will show that
  \begin{equation*}
    \phi(r) \leq \epsilon r^{n+2\beta}
  \end{equation*}
  with $\epsilon$ as in the proposition.
  The assertion then follows from \autoref{Thm_MorreyEmbedding}.

  Trivially, we have
  \begin{equation*}
    \phi(r) \leq \epsilon r^{n-2+2\alpha}.
  \end{equation*}
  The following proposition strengthens this estimate using \eqref{Eq_PDE}.

  \begin{prop}
    \label{Prop_PhiGrowth}
    For $0 < s \leq r \leq R$ and if $\alpha \leq 1$, we have
    \begin{equation*}
      \phi(s) \leq c\(\frac{s}{r}\)^{n+2}\phi(r) + \epsilon r^{n-2+3\alpha}.
    \end{equation*}
  \end{prop}

  We will postpone the proof for a short while to explain how the proof of \autoref{Prop_HildebrandtAPrioriEstimate} is completed.  
  To improve the exponent we use the following lemma, whose proof is very simple and deferred to the end of this section.

  \begin{lemma}
    \label{Lem_BootstrapExponent}
    If $\phi \co [0,R] \to [0,\infty)$ is a non-decreasing function and $c,\epsilon>0$, $\alpha > \beta > 0$ are constants such that for all $0 < s\leq r \leq  R$
    \begin{equation*}
      \phi(s) \leq c\(\frac{s}{r}\)^\alpha\phi(r) + \epsilon r^\beta,
    \end{equation*}
    then we have
    \begin{equation*}
      \phi(r)\lesssim_{c,\alpha,\beta} \(\frac{\phi(R)}{R^\beta} + \epsilon\)r^\beta.
    \end{equation*}    
  \end{lemma}

  We derive that
  \begin{equation*}
    \Abs{\nabla f}_{\cL^{2,n-2+2\alpha'}(V)} \leq \epsilon
  \end{equation*}
  with $\alpha' = \frac{3}{2}\alpha$.
  If $\alpha' < 1$, then by \autoref{Prop_CampanatoMorrey} we have
  \begin{equation*}
    \Abs{\nabla f}_{L^{n-2+2\alpha',2}(V)}
    \leq \epsilon
  \end{equation*}
  and we can restart the argument with $\alpha'$ instead of $\alpha$ and $V$ instead of $U$.
  Iterating this a finite number of times we will eventually end up in the case $\alpha' > 1$.
  In this case
  \begin{equation*}
    \phi(r) \leq \epsilon r^{n+2\beta}
  \end{equation*}
  with $\beta = \frac{\alpha' - 1}{2}$.
  This completes the proof.
\end{proof}

\begin{proof}[Proof of \autoref{Prop_PhiGrowth}]
  Fix a ball $B_r(x) \subset U$ with centre $x \in V$.
  We may assume that $f(x) = 0$, because in all that follows we can work with $f-f(x)$ instead.
  
  \setcounter{step}{0}
  \begin{step}  
    We can write $f = g + h$ with $g,h \co \bar B_r(x) \to \R^k$ satisfying
    \begin{equation}
      \label{Eq_PDEForG}
      \Delta g = A + B(\nabla f) + C(\nabla f \otimes \nabla f) \qandq
      g|_{\del B_r(x)} = 0
    \end{equation}
    and
    \begin{equation*}
      \Delta h = 0 \qandq
      h|_{\del B_r(x)} = f|_{\del B_r(x)}.
    \end{equation*}
  \end{step}

  \begin{step}
    We have
    \begin{equation*}
      \Abs{g}_{L^\infty(B_r(x))} \leq \epsilon r^\alpha \qandq
      \Abs{h}_{L^\infty(B_r(x))} \leq \epsilon r^\alpha.
    \end{equation*}
  \end{step}  

  By \autoref{Thm_Poincare} and \autoref{Thm_MorreyEmbedding} we have $[f]_{C^{0,\alpha}(U)} \leq \epsilon$.
  From $f(x) = 0$ it follows that $\Abs{f}_{L^\infty(B_r(x))} \leq \epsilon r^\alpha$.
  The maximum principle implies the asserted bound on $h$; the bound on $g$ then follows.

  \begin{step}
    We have
    \begin{equation*}
      \int_{B_r(x)} \abs{\nabla g}^2 \leq \epsilon r^{n-2+3\alpha}.
    \end{equation*}
  \end{step}

  Since $g$ vanishes on $\del B_r(x)$ and using \eqref{Eq_PDEForG},
  \begin{align*}
    \int_{B_r(x)} \abs{\nabla g}^2
    &= \int_{B_r(x)} \inner{\Delta g}{g} \\
    &\lesssim \int_{B_r(x)} \abs{g}(1+\abs{\nabla f}^2) \\
    &\leq \epsilon r^{n-2+3\alpha}.
  \end{align*}

  \begin{step}
    For $s \leq r$, we have
    \begin{equation*}
      \int_{B_s(x)} \abs{\nabla h - \overline{\nabla h}_{x,s}}^2
      \lesssim
      \(\frac{s}{r}\)^{(n+2)}\int_{B_r} \abs{\nabla h - \overline{\nabla h}_{x,r}}^2.
    \end{equation*}
  \end{step}
  
  This is \autoref{Thm_CampanatoEstimate} for $\nabla h$.
  
  \begin{step}
    We prove the proposition.
  \end{step}

  Using the preceding steps, we compute
  \begin{align*}
    \int_{B_s(x)} \abs{\nabla f - \overline{\nabla f}_{x,s}}^2
    &\leq
      \int_{B_s(x)} \abs{\nabla h - \overline{\nabla h}_{x,s} + \nabla g}^2 \\
    &\lesssim
      \int_{B_s(x)} \abs{\nabla h - \overline{\nabla h}_{x,s}}^2
      +
      \int_{B_s(x)} \abs{\nabla g}^2 \\
    &\lesssim
      \(\frac{s}{r}\)^{n+2}\int_{B_r(x)} \abs{\nabla h - \overline{\nabla h}_{x,r}}^2
      +
      \int_{B_r(x)} \abs{\nabla g}^2 \\
    &\lesssim
      \(\frac{s}{r}\)^{n+2}\int_{B_r(x)} \abs{\nabla f - \overline{\nabla f}_{x,r}}^2
      +
      \epsilon r^{n-2+3\alpha}.
  \end{align*}
  Taking the supremum over $x \in V$ yields the asserted statement.
\end{proof}

\begin{proof}[Proof of \autoref{Lem_BootstrapExponent}]
  This is similar to but somewhat simpler than \cite[Lemma 3.4]{Han2011}.
  If we choose $\tau < 1$ such that $\gamma := c\tau^{\alpha-\beta} < 1$,
  then
  \begin{align*}
    \phi(\tau^k R)
    &\leq \gamma \phi(\tau^{k-1}R)\tau^\beta + \frac{\epsilon}{\tau^\beta} (\tau^k R)^\beta \\
    &\leq \(\gamma^k \frac{\phi(R)}{R^\beta} + \frac{\epsilon}{(1-\gamma)\tau^{\beta}}\)(\tau^kR)^\beta.
  \end{align*}
  From this the assertion follows immediately.
\end{proof}


\section{Morrey and Campanato spaces}
\label{Sec_MorreyCampanato}

An excellent exposition of Morrey and Campanato spaces can be found in Struwe's lecture notes \cite[Kapitel 8 and 10]{Struwe2014}.
We only state the definitions and the results we make use of.

Assume $U \subset \R^n$ is open with smooth boundary.
Let $1 \leq p < \infty$ and $\lambda \geq 0$.

\begin{definition}
  \label{Def_Morrey}
  The \defined{Morrey space} $(L^{p,\lambda}(U),\Abs{\cdot}_{L^{p,\lambda}(U)})$ is the normed vector space defined by
  \begin{equation*}
    L^{p,\lambda}(U)
    := \set*{ f \in L^p(U) : \Abs{f}_{L^{p,\lambda}(U)} < \infty }
  \end{equation*}
  and
  \begin{equation*}
    \Abs{f}_{L^{p,\lambda}(U)}
    := \sup_{x\in U, r > 0} \(r^{-\lambda}\int_{B_r(x) \cap U} \abs{f}^p\)^{1/p}.
  \end{equation*}
\end{definition}

\begin{definition}
  The \defined{Campanato space} $(\cL^{p,\lambda}(U),\Abs{\cdot}_{\cL^{p,\lambda}(U)})$ is the normed vector space defined by
  \begin{equation*}
    \cL^{p,\lambda}(U)
    := \set*{ f \in L^p(U) : [f]_{\cL^{p,\lambda}(U)} < \infty }
  \end{equation*}
  and
  \begin{equation*}
    \Abs{f}_{\cL^{p,\lambda}(U)}
    := \Abs{f}_{L^p(U)} + [f]_{\cL^{p,n}(U)}.
  \end{equation*}
  Here the \defined{Campanato semi-norm} is defined by
  \begin{equation*}
    [f]_{\cL^{p,\lambda}(U)}
    := \sup_{x\in U, r > 0} \(r^{-\lambda}\int_{B_r(x) \cap U} \abs{f-\bar f_{x,r}}^p\)^{1/p}
  \end{equation*}
  with
  \begin{equation*}
    \bar f_{x,r}
    := \fint_{B_r(x)\cap U} f.
  \end{equation*}
\end{definition}

Both Morrey and Campanato spaces are Banach spaces.
The following shed some more light on the relation between Morrey, Campanato and Hölder spaces, and the Campanato regularity properties of harmonic functions.

\begin{prop}[{\cite[Lemma 10.3.1]{Struwe2014}}]
  \label{Prop_CampanatoMorrey}
  If $\lambda \leq n$, then for all $f \in \cL^{p,\lambda}(U)$ we have
  \begin{equation*}
    \Abs{f}_{L^{p,\lambda}(U)} \lesssim \Abs{f}_{\cL^{p,\lambda}(U)}.
  \end{equation*}
\end{prop}

\begin{theorem}[Poincaré inequality]
  \label{Thm_Poincare}
  For all $f \in L^{p,\lambda}(U)$, we have
  \begin{equation*}
    [f]_{\cL^{p,\lambda+p}(U)}
    \lesssim \Abs{\nabla f}_{L^{p,\lambda}(U)}.
  \end{equation*}  
\end{theorem}

\begin{theorem}[Morrey embedding {\cite[Satz 8.6.5]{Struwe2014}}]
  \label{Thm_MorreyEmbedding}
  For all $f \in \cL^{p,n+p\alpha}(U)$, we have
  \begin{equation*}
     [f]_{C^{0,\alpha}(\bar U)}
     \lesssim [f]_{\cL^{p,n+p\alpha}(U)}
  \end{equation*}
\end{theorem}

\begin{theorem}[{\cite[Lemma 10.2.1]{Struwe2014}} and {\cite[Lemma 3.10]{Han2011}}]
  \label{Thm_CampanatoEstimate}
  If $f \in W^{1,2}(B_r(x))$ satisfies
  \begin{equation*}
    \Delta f = 0
  \end{equation*}
  and $0 < s < r$, then
  \begin{equation*}
    \int_{B_s(x)} \abs{f - \bar f_{x,s}}^2
    \lesssim
    \(\frac{s}{r}\)^{(n+2)}\int_{B_r(x)} \abs{f - \bar f_{x,r}}^2.
  \end{equation*}
\end{theorem}


\paragraph{Acknowledgements.}

AJ would like to thank S.-T.~Yau for introducing him to this problem.
We thank Song Sun for pointing out a mistake in an earlier proof of \autoref{Prop_BlowUpSingularSet}.

\printreferences

\end{document}